\documentclass[12pt]{amsart}
\usepackage[margin=1in]{geometry}
\usepackage{fancyhdr}
\usepackage{hyperref}
\usepackage{graphicx,amssymb}
\usepackage{subfigure}
\usepackage{float}
\usepackage{xcolor}
\usepackage{cite}
\usepackage{amsmath,amsthm,amsfonts,amssymb,bm}
\usepackage{setspace}
\newtheorem{theorem}{Theorem}[section]

\newtheorem{theorem1}{Theorem}[section] 
\newtheorem{corollary}{Corollary}[section]
\newtheorem{lemma}{Lemma}[section]
\newtheorem{remark}{Remark}[section]
\numberwithin{equation}{section}
\numberwithin{theorem}{section}

\makeatletter
\def\tagform@#1{\maketag@@@{\ignorespaces#1\unskip\@@italiccorr}}
\let\orgtheequation\theequation
\def\theequation{(\orgtheequation)}
\makeatother

\allowdisplaybreaks
\everymath{\displaystyle}
\begin{document}
	\title{Stability and Instability on the De Gregorio Modification  of the Constantin-Lax-Majda model }
	\begin{abstract}
		The Constantin-Lax-Majda (CLM) model and the De Gregorio model which is a modification of the CLM model are well-known for their ability to emulate the behavior of the 3D Euler equations, particularly their potential to develop finite-time singularities. The stability properties of the De Gregorio model on the torus near the ground state $-\sin\theta$ have been well studied. However, the stability analysis near excited states $-\sin k\theta$ with $k\ge 2$ remains challenging.  This paper focuses on analyzing the stability and instability of the  De Gregorio model on torus around the first excited state $-\sin 2\theta$.  The linear and nonlinear instability are established for a broad class of initial data, while nonlinear stability is proved for another large class of initial data in this paper. Our analysis  reveals that  solution behavior to the De Gregorio model near excited states demonstrates different stability patterns depending on initial conditions. One of new ingredients in our  instability analysis involves deriving a second-order ordinary differential equation (ODE) governing the Fourier coefficients of solutions and examining the spectral properties of a positive definite quadratic form emerging from this ODE. The approach of this paper would be applicable to other related models and problems.		
	\end{abstract}
	\author{Jie Guo$^{*}$\,and\   QuanSen Jiu$^{\dagger}$}
	
	\address{$^{*}$School of Mathematical Sciences, Capital Normal University, Beijing, 100048, P.R. China.}
	
	\email{2230501023@cnu.edu.cn}
	\address{$^{\dagger}$School of Mathematical Sciences, Capital Normal University, Beijing, 100048, P.R. China.}
	
	\email{jiuqs@cnu.edu.cn}
	
	\maketitle
	
	\medskip
	{\bf Key words:}
	CLM model, De Gregorio model, stability, instability, excited states.

	{\bf MSC 2020:} 
	35B10, 35B35, 35C10, 35Q35. 

	\section{Introduction}
	The classical Constantin-Lax-Majda (CLM) model is
	\begin{equation}
		\partial_{t}\omega = \omega H\omega,
		\label{CLM}
	\end{equation}
	where $\omega=\omega(t,x)$ is an unknown function, which is defined on $\mathbb{R}_{+}\times\Omega$. Here, $\Omega$ can be either the entire real line $\mathbb{R}$ or a torus $\mathbb{T}$. For simplicity, we take $\mathbb{T}=[-\pi,\pi]$. And $H$ represents the Hilbert transform which is
	\begin{align}
		H\omega(\theta) = \frac{1}{2\pi}P.V.\int_{-\pi}^{\pi}\cot\frac{\theta-\phi}{2}\omega(\phi)d\phi
	\end{align}
	if $\omega$ is defined on the torus $\mathbb{T}$, and
	\begin{align}
		H\omega(\theta) = \frac{1}{\pi}P.V.\int_{-\infty}^{\infty}\frac{\omega(\phi)}{\theta-\phi}d\phi
	\end{align}
	if $\omega$ is defined on the whole line $\mathbb{R}.$
	The CLM model is renowned for analyzing potential singularities in the three-dimensional Euler equations, capturing the essential dynamics of the three-dimensional mechanism (see \cite{CL1985}). De Gregorio \cite{DS1996} suggested to include a convection term to the
	Constantin-Lax-Majda model. The modified equation is expressed as
	\begin{equation}\label{DG}
		\left\{
		\begin{aligned}
			&\partial_{t}\omega + u\partial_{\theta}\omega = \omega\partial_{\theta}u,\\
			&\partial_{\theta}u = H\omega,
		\end{aligned}\right.
	\end{equation}
	where $H$ represents the Hilbert transform.  Since $u$ is determined by $\omega$ only up to a constant, \ref{DG} is still incomplete, certain gauge condition on $u$ such as $\int_{\mathbb{T}}ud\theta = 0$ or $u(t,0) = 0$ is imposed in \cite{JS2019}. And Jia, Stewart and Sverak verified that the solutions to \ref{DG} under different gauges are equivalent in \cite{JS2019}. By interpolating the CLM model and the De Gregorio model, Okamoto, Sakajo and Wunsch \cite{OS2008} further introduced the following one-parameter family of models
	\begin{equation} \label{gCLM}
		\left\{
		\begin{aligned}
			&\partial_{t}\omega + au\partial_{\theta}\omega = \omega\partial_{\theta}u,\\
			&\partial_{\theta}u = H\omega.
		\end{aligned}
		\right.
	\end{equation}
	Here, $a$ is a parameter and $H$ represents the Hilbert transform. If $a = 0,$ \ref{gCLM} reduces to the CLM model \ref{CLM}. If $a = 1,$ it becomes the De Gregorio model \ref{DG}.

	It was shown   that the De Gregorio model \ref{DG} on the whole line $\mathbb{R}$ has a finite time singularity for some smooth odd initial data in \cite{CH2021} and  for initial data with lower regularity in \cite{EJ2020}. Further studies on the \ref{DG} and its modifications are referred to \cite{H2024,C2020,C2021,AL2024,WM2011,F2023} and references therein. Moreover, the role of the convection term was addressed  in \cite{HL2009,OH2013,EJ2020} and references therein.
	
	Numerical simulations indicate that, for general smooth initial data, the solutions of the De Gregorio model \ref{DG} tend to  ground states $A \sin(\theta-\theta_{0})$ as $t\rightarrow \infty$ (see \cite{OS2008}). Jia, Steward and Sverak  \cite{JS2019} first investigated rigorously the stability of the De Gregorio model \ref{DG} around the ground state $-\sin\theta$, by using a profound spectral analysis approach.  In \cite{LL2020}, Lei, Liu and Ren established the global well-posedness of \ref{DG} for general initial data with non-negative (or non-positive) vorticity, based on the discovery of a new conserved quantity.
	The  stability of the De Gregorio model \ref{DG} around the ground state $-\sin\theta$   was shown in \cite{LL2020} as well, by introducing a  crucial basis  to analyze the corresponding linearized equation and introducing an effective and weighted Hilbert space  to study the  nonlinear stability.  However, the study on the stability around  the first excited state $-\sin2\theta$ or other excited states $-\sin k\theta$ for $k\ge 3$ remains a difficult problem, as mentioned in \cite{JS2019, LL2020}.

	In this paper, we are concerned with the stability and instability of the De Gregorio model \ref{DG}  around   the first excited state $-\sin2\theta$ on the torus $\mathbb{T}$. To this end, we define
	\begin{align}
		\omega = -\sin2\theta + \eta,\ u = \frac{1}{2}\sin2\theta + v.
	\end{align}
	The perturbations $\eta,v$ then satisfy
	\begin{equation}\label{perturbation equ}
		\left\{\begin{split}
			&\partial_{t}\eta + L\eta =N(\eta),\\
			&\partial_{\theta}v = H\eta.\\
		\end{split}\right.
	\end{equation}
	Then the linearized equation to \ref{DG} for $\eta$ and $v$ reads as
	\begin{equation}\label{etav}
		\left\{
		\begin{aligned}
			& \eta_{t} + L\eta = 0,\\
			& \partial_{\theta}v = H\eta,\\
		\end{aligned}\right.
	\end{equation}
	where $L$ is a linearized operator
	\begin{align}\label{operator}
		L\eta = \frac{1}{2}\sin2\theta\partial_{\theta}\eta - \cos2\theta\eta + \sin2\theta H\eta - 2\cos2\theta v
	\end{align}
	and $N(\eta)$ is the nonlinear term
	\begin{align}
		N(\eta)=\partial_{\theta}v\eta-v\partial_{\theta}\eta.
	\end{align}
	Since \eqref{etav} is still incomplete, we impose  initial data and a gauge condition respectively, which are
	\begin{align}\label{1228}	
		\eta(0,\theta) = \eta_{0}(\theta),\quad v(t,0)=0.
	\end{align}
	Motivated by \cite{LL2020}, we set
	\begin{align}\label{tek}
		\tilde{e}_{k}^{(o)} = \frac{e_{k+2}^{(o)}}{k+2}-\frac{e_{k}^{(o)}}{k},\quad k\geq 1,
	\end{align}
	where $e_{k}^{(o)}=\sin k\theta,$  and define
	\begin{align}\label{HDW}
		\mathcal{H}_{DW} = \big\{\eta\in H^{1}(\mathbb{T})| \eta\ is\ odd, \int_{\mathbb{T}}\frac{|\partial_{\theta}\eta|^{2}}{|\sin\theta|^{2}}d\theta<\infty\big\}.
	\end{align}
	Then  $(\mathcal{H}_{DW},\rho)$ is a Hilbert space with inner product defined as
	\begin{align}\label{proHDW}
		\left\langle\xi,\eta\right\rangle_{\rho} = \int_{-\pi}^{\pi}\rho\partial_{\theta}\xi\partial_{\theta}\eta d\theta,
	\end{align}
	where $\rho = \frac{1}{4\pi\sin^{2}\theta}$.
	It can be proved that
	$\{\tilde{e}_{k}^{(o)}, k\geq 1\}$ is a complete orthonormal basis for $\mathcal{H}_{DW}$ (see Lemma \ref{lebasis}) and hence the functions in $\mathcal{H}_{DW}$ can be expressed by the linear combinations of $\{\tilde{e}_{k}^{(o)}, k\geq 1\}$.

	Before state our main results, we introduce some notations. Denote
	\begin{align*}
		&L^{2}(\mathbb{T}) = \left\{f| f\in L^{2}(-\pi,\pi), f\ is\ periodic\ on\  [-\pi,\pi]\right\},\\
		&H^{m}(\mathbb{T}) = \left\{f| f^{(k)}\in L^{2}(-\pi,\pi), f^{(k)}\ is\ periodic\ on\  [-\pi,\pi], for\ all\ 0\leq k\leq m\right\},
	\end{align*}
	where $f^{(k)}$ means the $k-$th order derivative of $f$.
	
	Throughout this paper, $\|\cdot\|_{L^{p}}, \|\cdot\|_{H^{m}}$ and $\|\cdot\|_{L^{\infty}}$ denote the norms of $L^{p}(\mathbb{T}), H^{m}(\mathbb{T})$ and $L^{\infty}(\mathbb{T})$, respectively.
	Similarly, $\left\langle\cdot,\cdot\right\rangle$ means the usual inner product in $L^{2}(\mathbb{T})$, that is
	\begin{align*}
		\left\langle f,g\right\rangle=\int_{-\pi}^{\pi}fg d\theta.
	\end{align*}
	Notice that the norm in the Sobolev space $\mathcal{H}_{DW}$ can be expressed as follows:
	\begin{align*}
		\|\eta\|_{\mathcal{H}_{DW}} = \left(\int_{-\pi}^{\pi}\rho(\partial_{\theta}\eta)^{2}d\theta\right)^{1/2}  = \|\rho^{1/2}\partial_{\theta}\eta\|_{L^2}.
	\end{align*}
	
	Now we are ready to present our main results. The first result is the global existence and uniqueness of \ref{etav} with odd initial data, which is
	\begin{theorem}\label{the existence}
		Assume that the initial data $\eta_{0}\in \mathcal{H}_{DW}$ and $\rho^{1/2}\partial_{\theta}\eta_{0}\in H^{m}$ with $m>3$ an integer. Then for any $T>0,$ there exists a unique classical and odd solution to \ref{etav} and \ref{1228} satisfying $\eta\in C([0,T];\mathcal{H}_{DW})$ and
		\begin{align*}
			\rho^{1/2}\partial_{\theta}\eta\in C([0,T]; H^{m}(\mathbb{T}))\cap C^{2}([0,T];H^{m-2}(\mathbb{T})).
		\end{align*}
	\end{theorem}
	To prove Theorem \ref{the existence}, we employ Galerkin's method through the basis $\{\tilde{e}_{k}^{(o)}\}$ and construct the approximate solutions $\eta_{n} = \sum_{k=1}^{n}\tilde{\eta}_{k}^{(o)}(t)\tilde{e}_{k}^{(o)}$. In the process of the proof, we obtain  uniform estimates for the Fourier coefficients of the solutions as a byproduct (see Corollary \ref{coro}), which will ensure the rigor of our subsequent proof.
	
	For more  general initial data, we can also establish the following result.
	\begin{theorem1}{\hspace{-.2cm}$^\prime$}\label{the gexistence}
		Asuume that $\eta_{0}\in H^{m}$ with $m>3$  an integer. Then for any $T>0$, there exits a unique classical solution to \ref{etav} and \ref{1228} satisfying
		\begin{align*}
			\eta\in C([0,T];H^{m})\cap C^{2}([0,T];H^{m-2}).
		\end{align*}
	\end{theorem1}
	The result is derived by constructing approximate solutions through general basis $\{\sin k\theta, k\geq 1\}\cup\{\cos k\theta, k\geq 0\}$. The proof is analogous to that of Theorem \ref{the existence}, so we omit the details here.

	Based on Theorem \ref{the existence}, we obtain the following   instability results to the linearized equations \ref{etav}  around the excited state $-\sin2\theta$  for some specific initial data  $\eta_0 = \omega_{0} + \sin2\theta.$
	\begin{theorem}\label{the linear instability}
		Suppose that the conditions in Theorem \ref{the existence} are satisfied and  $\left\langle-L\eta_0, \eta_0\right\rangle_\rho\ge 0$ for $\eta_0\neq 0$. Then there exist   two absolute constants  $\lambda_{2}>\lambda_{1}>0$ such that the solution presented in Theorem \ref{the existence} satisfies
		\begin{align}\label{instability25}
			0< J_{1}^{1/2}(t)< \|\eta\|_{\mathcal{H}_{DW}} < J_{2}^{1/2}(t),
		\end{align}
		for $0<t<\infty$, where
		\begin{align}\label{Ji}
			J_{i}(t) = \frac{\left\langle\eta_0, \eta_0\right\rangle_{\rho} + \frac{1}{\sqrt{ \lambda_i}}\left\langle -L \eta_0, \eta_0\right\rangle_{\rho}}{2} e^{2 \sqrt{ \lambda_{i} }t} + \frac{\left\langle\eta_0, \eta_0\right\rangle_{\rho} - \frac{1}{\sqrt{ \lambda_i}}\left\langle -L \eta_0, \eta_0\right\rangle_{\rho}}{2} e^{-{2 \sqrt{\lambda_{i} }t}}, i=1,2.
		\end{align}
	\end{theorem}
	\begin{remark}\label{rmk1}
		The  constants $\lambda_1$ and $\lambda_2$ are two absolute and positive constants satisfying $\frac{1}{50}<\lambda_1<\lambda_2<\frac{3}{5}$, see Lemma \ref{lemmabound} in Appendix for more details.
	\end{remark}
	
	\begin{corollary}\label{coro linear instability} If one of the following two conditions is satisfied:
		
		(i) $\eta_{0}= a_{1}\tilde{e}_{1}^{(o)}$ with $a_{1}\neq 0$;
		
		(ii) $\eta_{0}= a_{1}\tilde{e}_{1}^{(o)} + a_{k}\tilde{e}_{k}^{(o)}$ with $k\geq 2$ an integer and
		\begin{align*}
			0<\ a_{k}^{2}\leq\ \frac{11}{18(d_{k+2}-d_{k})}a_{1}^{2},
		\end{align*}
		
		then the solution presented in Theorem \ref{the existence} satisfies \eqref{instability25}.
	\end{corollary}
	
	Thanks to Theorem \ref{the linear instability} and Corollary \ref{coro linear instability}, by appropriately selecting the initial data, one can derive an instability result for the nonlinear problem \eqref{perturbation equ}. In particular, the steady state $(-\sin 2\theta,\frac{1}{2}\sin 2\theta)$ of \ref{perturbation equ} and \ref{1228} is unstable under the Lipschitz structure.  More precisely, we have 
	\begin{theorem}\label{the nonlinear instability}
		For any integer $m>3$ and $\delta>0, K>0,$ suppose that the function $F:[0,\infty)\to\mathbb{R}$ satisfies
		\begin{align}\label{def F}
			F(y)\leq Ky,\ for\ any\ y\in[0,\infty).
		\end{align}
		Then there exist initial data $\eta_{0}$ such that
		\begin{align*}
			\|u_{0}\|_{H^{m}}<\delta, \quad 0\leq\left\langle L_{1}u_{0},u_{0}\right\rangle\leq\sqrt{\lambda_{1}}\|u_{0}\|_{L^{2}},
		\end{align*}
		where $u_{0}=\rho^{1/2}\partial_{\theta}\eta_{0}$, $L_{1}u_{0}=-\rho^{1/2}\partial_{\theta}L\eta_{0}$ and $\lambda_1$ is same as in Remark \ref{rmk1}. The following hold:
		
		(1) There exists  $T_0>0$ such that the unique classical solution to \eqref{perturbation equ} and \eqref{1228} satisfying 
		\begin{align*}
			\rho^{1/2}\partial_{\theta}\eta\in C([0,T_0]; H^{m}(\mathbb{T}))\cap C^{2}([0,T_0];H^{m-2}(\mathbb{T})).
		\end{align*}
		where  $T_0=\frac{1}{C_{{1}}}\ln(1+\frac{C_{{1}}}{2C_{{2}}\|u_{0}\|_{H^{m}}})$ with absolute constants $C_{1},C_{2}>0$  (see \ref{C12} for details).
		
		(2) The solution $\eta$ further satisfies
		\begin{align}\label{1.8}
			\|u(t_{K})\|_{L^{2}}>F(\|u_{0}\|_{H^{m}}),
		\end{align}
		for some $ t_{K}\in(0,T_{0}),$ where $u=\rho^{1/2}\partial_{\theta}\eta.$
	\end{theorem}
	\begin{remark}
		Theorem \ref{the nonlinear instability} establishes a nonlinear  instability to \ref{perturbation equ} in the sense of Lipschitz structure, which implies that the solution to \eqref{perturbation equ} does not exhibit the following stability property: there exists a constant $C>0$ such that
		\begin{align}\label{Lip instability}
			\mathop{\sup}_{0< t\leq T}\|u\|_{L^{2}}\leq C\|u_0\|_{H^{m}},
		\end{align}
		for any $T>0$,	where $u=\rho^{1/2}\partial_{\theta}\eta$.  Such an estimate \eqref{Lip instability} is often anticipated in the context of global stability theory. The proof of Theorem \ref{the nonlinear instability} is inspired by \cite{JJ2013}. Ideally, one might expect a stronger form, replacing the right-hand side of \eqref{1.8} by a fixed constant $\varepsilon>0$, thereby implying that arbitrarily small initial data could lead to solutions eventually escapes any ball of radius $\varepsilon>0$, which is an instability in the sense Hardamard. However, proving such a strong form of instability remains a challenging problem, which we leave for future investigation.
	\end{remark}
	
	\begin{remark}
		Theorem \ref{the nonlinear instability} holds for a broad class of initial data. We present a simple example of initial data as follows. Let
		\begin{align*}
			\eta_{0}=a_{1}\tilde{e}_{1}^{(o)}+a_{k}\tilde{e}_{k}^{(o)}.
		\end{align*}
		Here $a_{1},a_{k}\neq 0$ for $k\geq 2$. It is straightforward to verify that $\eta_{0}\in\mathcal{H}_{DW}$ and $\rho^{1/2}\partial_{\theta}\eta_{0}\in H^{m}$ for $m>3$. If the coefficients  of $\eta_{0}$ satisfy
		\begin{align}
			\frac{\frac{11}{18}-\sqrt{\lambda_{1}}}{\sqrt{\lambda_{1}}+(d_{k+2}-d_{k})}a_{1}^{2}\leq a_{k}^{2}\leq \frac{11}{18(d_{k+2}-d_{k})}a_{1}^{2},
		\end{align}
		then it follows that
		\begin{align*}
			0\leq\left\langle-L\eta_0, \eta_0\right\rangle_{\rho}\leq\sqrt{\lambda_{1}}\left\langle\eta_0, \eta_0\right\rangle_{\rho}.
		\end{align*}
		The constants $\lambda_1$ is an absolute and positive constants satisfying $\frac{1}{50}<\lambda_1<\frac{3}{5}$, and $d_{k}-d_{k+2}<0$ for all $k\geq 2$. See Lemma \ref{lemmabound} in Appendix for more details.
	\end{remark}
	
	Lastly, we establish global well-posedness and nonlinear stability with  initial data of the form $\eta_0=\sum_{k\geq 1}a_{2k}\tilde{e}_{2k}^{(o)}$ on the torus $\mathbb{T}$, which is
	\begin{theorem}\label{wellp}
		Suppose that $\eta_{0}=\sum_{k\geq 1}a_{2k}\tilde{e}_{2k}^{(o)}$. Then there exists $\varepsilon>0$ such that if $\|\eta_{0}\|_{\mathcal{H}_{DW}}\leq\varepsilon$,  the  equation \eqref{perturbation equ} with \eqref{1228} is globally well-posed. Moreover, it holds that
		\begin{align*}
			\|\eta\|_{\mathcal{H}_{DW}} \lesssim e^{-\frac{3}{8}t}\|\eta_{0}\|_{\mathcal{H}_{DW}}
		\end{align*}
		for all $t\geq 0$.
	\end{theorem}
	
	To prove Theorem \ref{the existence}, we make use of the Galerkin's method. Since we will need the sufficient decay of the coefficient functions $\tilde{\eta}_{k}^{(o)}(t),  \frac{d}{dt}\tilde{\eta}_{k}^{(o)}(t)$ and $\frac{d^{2}}{dt^{2}}\tilde{\eta}_{k}^{(o)}(t)$ with respect to $k$ when $k\ge 1$ large enough in the subsequent stability analysis, we construct the approximate solutions through the basis $\{\tilde{e}_{k}^{(o)}, k\geq 1\}$ and make a priori and uniform estimates up to higher order derivatives.  The proof of Theorem \ref{the gexistence}$^\prime$ is similar to that of Theorem \ref{the existence}.   To establish Theorem \ref{the linear instability}, we note that the solution to \ref{etav} can be expressed as $\eta(t,\theta)=\sum_{k\geq 1}\tilde{\eta}_{k}^{(o)}(t)\tilde{e}_{k}^{(o)}$. A straightforward computation yields
	\begin{align}\label{ODE}
		\frac12 \frac{d}{dt}\sum_{k\geq 1}\left(\tilde{\eta}_{k}^{(o)}(t)\right)^2=\left\langle -L\eta, \eta\right\rangle_{\rho}=\sum_{k\geq 1}(-d_{k+2}+d_{k})\left(\tilde{\eta}_{k}^{(o)}(t)\right)^2,
	\end{align}
	where
	\begin{align*}
		-d_{k+2}+d_{k}= -\frac{k^{4}+4k^{3}+8k^{2}-8k-16}{2(k+2)^{2}k^{2}}
	\end{align*}
	for $k=1,2,\cdots$. It is  difficult to  analyze the  stability or instability on the linearized equation \ref{etav} since  the coefficients $-d_{k+2}+d_{k}$ may change signs. For instance, when $k=1$,  $-d_{3}+d_{1}=\frac{11}{18}>0$, and when $k\ge 2$, $-d_{k+2}+d_{k}<0$. 	 To overcome the difficulty, we analyze  a second-order ordinary differential equation (ODE) obtained from \eqref{ODE}, which can be written as
	\begin{eqnarray}\label{ODE2}
		\begin{split}
			&\frac12\frac{d^2}{dt^2}\sum_{k\geq 1}^{\infty}\left(\tilde{\eta}_{k}^{(o)}(t)\right)^2=\frac{d}{dt}\left\langle -L\eta,\eta\right\rangle_{\rho}\\
			&=\sum_{k=1}^{\infty}\frac{d}{dt}\left(d_{k}-d_{k+2}\right)\left(\tilde{\eta}_{k}^{(o)}(t)\right)^{2}.
		\end{split}
	\end{eqnarray}
	More precisely, concerning the right-hand side of \eqref{ODE2},  we consider
	\begin{eqnarray}\label{0323}
		\begin{split}
			&\sum_{k=1}^{n}\frac{d}{dt}\left(d_{k}-d_{k+2}\right)\left(\tilde{\eta}_{k}^{(o)}(t)\right)^{2}\\
			=&  \left(-d_{3} + d_{1}\right)^{2}\left(\tilde{\eta}_{1}^{(o)}\right)^{2}  + \left(-d_{4} + d_{2}\right)^{2}\left(\tilde{\eta}_{2}^{(o)}\right)^{2} +\sum_{k=1}^{n-2}f_{k} \\
			&+\left(-d_{n+1} + d_{n-1}\right)^{2}\left(\tilde{\eta}_{n-1}^{(o)}\right)^{2}  + \left(-d_{n+2} + d_{n}\right)^{2}\left(\tilde{\eta}_{n}^{(o)}\right)^{2}+ R_{n-1}+R_{n},
		\end{split}
	\end{eqnarray}
	where the quadratic form  $f_{k}$ is defined as
	\begin{equation*}\label{form}
		f_{k} = (-d_{k+2}+d_{k})^2\left(\tilde{\eta}_{k}^{(o)}\right)^{2} + 2(-2d_{k+2}^{2}+d_{k}d_{k+2}+ d_{k+2}d_{k+4})\tilde{\eta}_{k}^{(o)}\tilde{\eta}_{k+2}^{(o)}+(-d_{k+4}+d_{k+2})^2\left(\tilde{\eta}_{k+2}^{(o)}\right)^{2}
	\end{equation*}
	for $k=1,2,\cdots,n-2$, and the remainder term $R_{k}$ is defined as
	\begin{equation*}\label{Rk}
		R_{k}(t) = 2d_{k+2}(-d_{k+2}+d_{k})\tilde{\eta}_{k}^{(o)}(t)\tilde{\eta}_{k+2}^{(o)}(t)
	\end{equation*}
	for $k=n-1,n$ (see \eqref{form} and\eqref{Rk}).
	We prove that each $f_k (k=1,2,\cdots)$ is a positive definite quadratic form  and provide its uniform lower and upper  bound with respect to  $k\ge 1$. Moreover, employing the decay estimates on the remainder terms  and letting $n\to \infty$ in \eqref{0323}, we can obtain
	\begin{align*}\label{12eta}
		4\lambda_{1}\|\eta(t)\|_{\mathcal{H}_{DW}}^{2} < \frac{d^{2}}{dt^{2}}\|\eta(t)\|_{\mathcal{H}_{DW}}^{2} < 4\lambda_{2}\|\eta(t)\|_{\mathcal{H}_{DW}}^{2}
	\end{align*}
	for $0< t<\infty$, where $\lambda_{1}$ and $\lambda_{2}$ are two positive absolute constants.  Then by applying the comparison theorem for the second-order ODE, we can finish the proof of Theorem \ref{the linear instability}.  To prove Theorem \ref{the nonlinear instability}, we first prove local existence and uniqueness of the classical solution to the nonlinear problem  \ref{perturbation equ} and discuss the existence interval on time  in a detail way. Then we derive some nonlinear energy estimates for the perturbed problem \ref{perturbation equ}, which make it possible to take  the limit in the scaled perturbed problem to obtain the corresponding linearized equation. With the help of the results established on the linear instability, we can obtain the instability of the nonlinear problem for a broad class of initial data in the sense of \ref{1.8}. Finally, for another large class of initial data, we establish the nonlinear stability  Theorem \eqref{wellp} of which proof  is based on the linear stability analysis, delicate estimates on the nonlinear terms and continuous argument.
	
	The organization of this paper is as follows. Section \ref{notation} presents preliminary material, including some basic lemmas and facts. In Section \ref{exist}, we establish Theorem \ref{the existence}, which concerns the global existence and uniqueness of solutions to the linearized equation \eqref{etav}. Section 4 contains the proof of Theorem \ref{the linear instability} and Corollary \ref{coro linear instability}, addressing the instability properties of solutions to the linearized equation \eqref{etav} near the excited state $-\sin2\theta$. Building upon these linear instability results, Section 5 demonstrates Theorem \ref{the nonlinear instability} which  concerns the instability analysis to the nonlinear problem \eqref{perturbation equ} for a broad class of initial data. In Section 6, we give  the proof of Theorems \ref{wellp} which is about the nonlinear stability to the  equation \ref{DG} for a large class of initial data.  Finally, in the Appendix we provide both rigorous analysis and numerical verification of the uniform positive lower and upper bounds for the quadratic forms $f_k$  as in \ref{form} for $k=1,2,\cdots$.
	
	Throughout the paper, we use $C,C_{i}$ to denote absolute constants and $C(A,B,\dots,Z)$ to denote constant depending on $A,B,\dots,Z$. These constants may vary from line to line, unless specified. The notation $A\lesssim B$ indicates that $A\leq CB$ for some positive constant $C$, which may vary on different lines.
	
	\section{Preliminaries}\label{notation}
	
	In this section, we present some basic lemmas and facts. We first prove that
	\begin{lemma}\label{lebasis}
		$\{\tilde{e}_{k}^{(o)}, k\geq 1\}$ is a complete orthonormal basis for $\mathcal{H}_{DW}$.
	\end{lemma}
	\begin{proof}[Proof of Lemma \ref{lebasis}]
		Notice that
		\begin{align*}
			\frac{\partial_{\theta}\tilde{e}_{k}^{(o)}}{\sin\theta}=-2\sin(k+1)\theta, \quad k\geq 1.
		\end{align*}
		
		It yields that
		\begin{align*}
			\big\langle\tilde{e}_{k}^{(o)},\tilde{e}_{l}^{(o)}\big\rangle_{\rho}=\delta_{kl}, \quad k,l\geq 1,
		\end{align*}
		where $\delta_{kl}=1$ if $k=l$ and $\delta_{kl}=0$ if $k\neq l$. Next we  show the completeness of $\{\tilde{e}_{k}^{(o)}, k\geq 1\}$. Assume $\xi\in\mathcal{H}_{DW}$, satisfying
		\begin{align*}
			\big\langle\xi,\tilde{e}_{k}^{(o)}\big\rangle_{\rho}=0
		\end{align*}
		for all $ k\geq 1$. Then it
		\begin{equation}\label{xi0}
			\int_{\mathbb{T}}\frac{\partial_{\theta}\xi}{\sin\theta}\sin(k+1)\theta d\theta =0
		\end{equation}
		for all $ k\geq 1$. Since the equality \ref{xi0} holds for $k=0$ as well. That is
		\begin{align*}
			\int_{\mathbb{T}} \partial_{\theta}\xi d\theta=0.
		\end{align*}
		Due to $\left\{\sin k\theta, k\geq 1\right\}$ forms an odd complete basis of $L^{2}(\mathbb{T})$, it concludes that $\frac{\partial_{\theta}\xi}{\sin\theta}=0$ which implies $\partial_{\theta}\xi = 0$. Thanks to oddness of $\theta$, we have that $\xi = 0$.
	\end{proof}
	The following is a comparison theorem on the second-order ordinary differential equation (see \cite{MM1949}).
	\begin{lemma}\label{Comp}
		Consider the differential equation
		\begin{align}
			y'' = p_{1}y'+p_{2}y+q, \quad x\geq x_{0},
		\end{align}
		where $p_{1}(x),p_{2}(x)$ and $q(x)$ are continuous functions when $x\geq x_{0}$, and let $y(x)$ be a solution of this equation such that
		\begin{align}
			y(x_{0})=y_{0},\quad y'(x_{0})=y'_{0}.
		\end{align}
		Suppose that there exits a solution of
		\begin{align}\label{cuode}
			u'' = p_{1}u'+p_{2}u,
		\end{align}
		such that
		\begin{align*}
			u(x)\neq 0,\quad x_{0}< x < x_{1}.
		\end{align*}
		Let $u_{0}(x)$ be the solution of \ref{cuode} such that $u_{0}(x_{0})=0, u'_{0}(x_{0})=1$ and let $X(x_{0})$ be the first zero of $u_{0}(x)$ to the right of $x_{0}$, if any such zero exists; otherwise let $X(x_{0})=+\infty.$\\
		(i) If $\phi(x)$ is such that
		\begin{align*}
			&\phi''>p_{1}\phi' + p_{2}\phi + q,\quad x\geq x_{0},\\
			&\phi(x_{0})=y(x_{0}),\quad \phi'(x_{0})=y'(x_{0}).
		\end{align*}
		Then
		\begin{align}\label{phiy}
			\phi(x) > y(x),\quad x_{0} <x \leq x_{1}.
		\end{align}
		(ii) The interval $x_{0}< x< X(x_{0})$ is the largest one in which the inequality \ref{phiy} can be asserted to hold.
	\end{lemma}
	
	Next, we present some basic properties  concerning with the Hilbert transform as follows (see \cite{BN1971,ZA2002} for instance).
	\begin{lemma}\label{Heo}
		For any $a > 0$, it holds
		\begin{align*}
			&H\sin(a\theta) = -\cos(a\theta),\\
			&H\cos(a\theta) = \sin(a\theta).
		\end{align*}
	\end{lemma}
	\begin{lemma}
		Let $f\in L^p$ with $1< p < \infty$, and assume $f$ is $2\pi$-periodic. Then the Fourier coefficient of the Hilbert transform $Hf$ at frequency $k$ is given by
		\begin{align*}
			\widehat{Hf}(k) = \{-i\ sgn k\}\hat{f}(k).
		\end{align*}
	\end{lemma}
	\begin{lemma}
		Let $f\in L^1$ and $2\pi$-periodic function such that $Hf\in L^1$ and is also $2\pi$-periodic. Then the conjugate Fourier series $\sum_{k=-\infty}^{\infty}\{-i\ sgn k\}\hat{f}(k)e^{ikx}$ of $f$ is the Fourier series of the Hilbert transform (conjugate function) $Hf$, i.e.,
		\begin{align*}
			Hf(x) \sim \sum_{k=-\infty}^{\infty}\{-i\ sgn k\}\hat{f}(k)e^{ikx}.
		\end{align*}
	\end{lemma}
	\begin{lemma}\label{le hilbert estimate}
		The Hilbert transform $H$ is a bounded linear operator from space $L^{p}$ to $L^{p}$ with $1<p<\infty$ and
		\begin{equation}
			\| Hf\|_{L^{p}}\leq C_{p}\| f\|_{L^{p}},
		\end{equation}
		for a constant $C_{p}>0$ depending on $p$.
		\label{HLp}
	\end{lemma}
	
	Finally, we prove
	\begin{lemma}
		For any integer $k\geq 1$, $\omega = \sin k\theta$ is a steady solution to \ref{DG} with $u(0)=0$.
	\end{lemma}
	\begin{proof}
		Applying the properties of Hilbert transform in the Lemma \ref{Heo}, we have
		\begin{align*}
			H(\sin k\theta) = -\cos k\theta,
		\end{align*}
		for $k\geq 1$. It follows from $u(0)=0$ that
		\begin{align*}
			u = -\frac{1}{k}\sin k\theta,
		\end{align*}
		which shows that $\omega = \sin k\theta\ (k=1,2,\cdots)$ is
		a steady solution to  \ref{DG} with $u(0)=0$.
	\end{proof}
	\begin{remark}
		As usual, $\omega = \sin \theta$ is called a ground state and $\omega = \sin k\theta$ with $k\ge 2$ are called excited states of  \ref{DG} (see \cite{JS2019,LL2020}).
	\end{remark}
	
	\section{Existence and uniqueness of linearized equation}\label{exist}
	In this section, we give the proof of Theorem \ref{the existence} and establish the global existence and uniqueness of equation \ref{etav} using Galerkin's method. To derive the decay rates for the coefficients $\tilde{\eta}_{k}^{(o)}(t), \frac{d}{dt}\tilde{\eta}_{k}^{(o)}(t)$ and $\frac{d^{2}}{dt^{2}}\tilde{\eta}_{k}^{(o)}(t)$ as required in Section \ref{sec linear instability}, we construct an approximate solution employing the basis  $\{\tilde{e}_{k}^{(o)}, k\geq 1\}.$ The proof of Theorem \ref{the gexistence}$^\prime$ can be proved in a similar way by using the basis   $\{\sin k\theta, k\geq 1\}\cup\{\cos k\theta, k\geq 0\}$ and we omit it here.
	
	{\it Proof of Theorem \ref{the existence}.}
	
	{\bf Step 1. Construction of the approximate solution.}
	
	\quad Fix a positive integer $n$, and define
	\begin{align}
		&\eta_{n} = \sum_{k=1}^{n}\tilde{\eta}_{k}^{(o)}(t)\tilde{e}_{k}^{(o)}\label{etans},\\
		&\partial_{\theta}v_{n} = H\eta_{n}\label{vn}.
	\end{align}
	We aim to determine the coefficients $\tilde{\eta}_k^{(o)}(t)$ such that
	\begin{align}\label{etan}
		\big\langle\partial_{t}\partial_{\theta}\eta_{n}, \partial_{\theta}\tilde{e}_{k}^{(o)}\rho\big\rangle + \big\langle\partial_{\theta}L\eta_{n}, \partial_{\theta}\tilde{e}_{k}^{(o)}\rho\big\rangle = 0
	\end{align}
	and
	\begin{align}\label{eta0}
		\big\langle\partial_{\theta}\eta_{0}, \partial_{\theta}\tilde{e}_{k}^{(o)}\rho\big\rangle= \tilde{\eta}_{k}^{(o)}(0)
	\end{align}
	for $0\leq t\leq T$, and $k=1,2,\cdots, n.$ In fact,  \ref{etan} can be rewritten as
	\begin{align}\label{retan}
		\big\langle\partial_{t}\sum_{l=1}^{n}\tilde{\eta}_{l}^{(o)}(t)e_{l+1}^{(o)}, e_{k+1}^{(o)}\big\rangle + \big\langle-\frac{1}{2\sin\theta}\partial_{\theta}L\eta_{n}, e_{k+1}^{(o)}\big\rangle = 0,
	\end{align}
	where
	\begin{align}\label{pleta}
		-\frac{1}{2\sin\theta}\partial_{\theta}L\eta_{n} = -\frac{1}{2}\cos\theta\partial_{\theta}^{2}\eta_{n} - 2\cos\theta\eta_{n} -\cos\theta H(\partial_{\theta}\eta_{n}) - 4\cos\theta v_{n}.
	\end{align}
	To simplify the notation, we set $	u_{n}:=-\sqrt{\pi}\rho^{1/2}\partial_{\theta}\eta_{n}$
	and define
	\begin{align*}
		L_{1}(u_{n})&:=-\sqrt{\pi}\rho^{1/2}\partial_{\theta}L\eta_{n}\\
		&=\frac{1}{2}\sin2\theta\partial_{\theta}u_{n}+\cos^{2}\theta u_{n}+2\cos\theta H(\sin\theta u_{n})-2\cos\theta\eta_{n}-4\cos\theta v_{n}.
	\end{align*}
	Then the ordinary differential equations \ref{etan} are equivalent to
	\begin{align}\label{un}
		\left\langle\partial_{t}u_{n}, \sin(k+1)\theta\right\rangle + \left\langle L_{1}(u_{n}), \sin(k+1)\theta\right\rangle = 0.
	\end{align}
	In fact, it can be verified that \ref{un} holds true for $k=0$, where we defined $\tilde{\eta}_{0}^{(o)}(t)=0.$
	According to the standard existence theory for the ordinary differential equations, there exists a unique absolutely continuous functions $\tilde{\eta}_{k}^{(o)}(t) (k=1,2,\cdots,n)$ satisfying \ref{etan} and \ref{eta0} for $0\leq t\leq T$.
	
	{\bf Step 2. Energy estimates.}~\\
	{ \it \underline{$L^{2}$ estimate of $u_{n}.$}}	Multiplying \ref{un} by $\tilde{\eta}_{k}^{(o)}(t),$ and summing up over $k=1,2,\cdots,n,$ we obtain
	\begin{align}
		\frac{1}{2}\frac{d}{dt}\| u_{n}\|_{L^{2}}^{2}& = -\left\langle L_{1}(u_{n}), u_{n}\right\rangle\nonumber\\
		&=-\big\langle-\frac{1}{2}\cos\theta
		\partial_{\theta}^{2}\eta_{n}-2\cos\theta\eta_{n}-\cos\theta H(\partial_{\theta}\eta_{n})-4\cos\theta v_{n}, u_{n}\big\rangle\\
		&=-\big\langle \frac{1}{2}\sin2\theta\partial_{\theta}u_{n} + \cos^{2}\theta u_{n}, u_{n}\big\rangle - \left\langle 2\cos\theta H(\sin\theta u_{n}), u_{n}\right\rangle\nonumber \\
		&\quad + \left\langle 2\cos\theta\eta_{n}, u_{n}\right\rangle + \left\langle 4\cos\theta v_{n}, u_{n}\right\rangle.\nonumber
	\end{align}
	Applying H\"{o}lder's inequality and the $L^{p}$ estimate of Hilbert transform, we obtain
	\begin{align}\label{cosHun}
		\left\langle 2\cos\theta H(\sin\theta u_{n}), u_{n}\right\rangle& \lesssim \ \| H(\sin\theta u_{n})\|_{L^{2}}\| u_{n}\|_{L^{2}}\nonumber\\ &\lesssim\ \| u_{n}\|_{L^{2}}^{2}.
	\end{align}
	Since $\eta_{n}(0)=0,$ we obtain
	\begin{align}
		\left\langle 2\cos\theta\eta_{n}, u_{n}\right\rangle
		&\lesssim\ \|\eta_{n}\|_{L^{\infty}}\| u_{n}\|_{L^{2}}\nonumber\\
		& \lesssim\ \|\partial_{\theta}\eta_{n}\|_{L^{2}}\| u_{n}\|_{L^{2}}\\
		&\lesssim\ \| u_{n}\|_{L^{2}}^{2}.\nonumber
	\end{align}
	Since $\int_{\mathbb{S}^1} v_{n} d\theta = \int_{\mathbb{S}^1} \eta_{n} d\theta = 0$, we can use Poincar\'{e}'s inequality, H\"{o}lder's inequality, $\partial_{\theta}v_{n}=H\eta_{n}$ and Lemma \ref{HLp} to obtain
	\begin{align}\label{cosv}
		\left\langle 4\cos\theta v_{n}, u_{n}\right\rangle &\lesssim\ \| v_{n}\|_{L^{\infty}}\| u_{n}\|_{L^{2}}\nonumber\\
		& \lesssim\ \|\eta_{n}\|_{L^{2}}\| u_{n}\|_{L^{2}}\\
		& \lesssim\ \|\partial_{\theta}\eta_{n}\|_{L^{2}}\| u_{n}\|_{L^{2}}\nonumber\\
		&\lesssim\ \| u_{n}\|_{L^{2}}^{2}.\nonumber
	\end{align}
	By collecting \ref{cosHun}-\ref{cosv}, we deduce that
	\begin{align*}
		\frac{d}{dt}\|u_{n}(t)\|_{L^{2}} \leq\ C\| u_{n}(t)\|_{L^{2}},
	\end{align*}
	and thus
	\begin{align}
		\mathop{\sup}_{0\le t\leq T}\|u_{n}(t)\|_{L^{2}} \leq\ C,
	\end{align}
	for any $T>0$, where $C$ is a positive constant independent of $n$.\\
	\\
	{\it \underline{$H^{1}$ estimate of $u_{n}.$}}	Multiplying equation \ref{un} by $(k+1)^{2}\tilde{\eta}_{k}^{(o)}(t)$, summing up over $k=1,2,\cdots,n,$ and using the integration by parts, we obtain
	\begin{align}
		\frac{1}{2}\frac{d}{dt}\|\partial_{\theta}u_{n}\|_{L^{2}}\ =\ -\left\langle\partial_{\theta}L_{1}(u_{n}), \partial_{\theta}u_{n}\right\rangle,
	\end{align}
	where
	\begin{align*}
		\partial_{\theta}L_{1}(u_{n}) = &\frac{1}{2}\sin2\theta\partial_{\theta}^{2}u_{n} + (\cos2\theta+\cos^{2}\theta)\partial_{\theta}u_{n} + \sin2\theta u_{n} - 2\sin\theta H(\sin\theta u_{n})\\
		&+ 2\cos\theta H(\cos\theta u_{n}+\sin\theta\partial_{\theta}u_{n})+2\sin\theta\eta_{n}+4\sin\theta v_{n}-4\cos\theta H\eta_{n}.
	\end{align*}
	Applying H\"{o}lder's inequality, Sobolev's inequality and Lemma \ref{HLp}, we have
	\begin{align}\label{sineta}
		\left\langle 2\sin\theta\eta_{n}, \partial_{\theta}u_{n}\right\rangle &\lesssim\ \|\eta_{n}\|_{L^{\infty}}\|\partial_{\theta}u_{n}\|_{L^{2}}\nonumber\\
		&\lesssim\ \| u_{n}\|_{H^{1}}^{2}.
	\end{align}
	Similarly, we have
	\begin{align}
		\left\langle 4\sin\theta v_{n}, \partial_{\theta}u_{n}\right\rangle &\lesssim\ \| v_{n}\|_{L^{\infty}}\|\partial_{\theta}u_{n}\|_{L^{2}}\nonumber\\
		&\lesssim\ \| H\eta_{n}\|_{L^{2}}\|\partial_{\theta}u_{n}\|_{L^{2}}\nonumber\\
		& \lesssim\ \|\eta_{n}\|_{L^{2}}\|\partial_{\theta}u_{n}\|_{L^{2}}\\
		&\lesssim\ \| u_{n}\|_{L^{2}}\|\partial_{\theta}u_{n}\|_{L^{2}}\nonumber\\
		&\lesssim\ \| u_{n}\|_{H^{1}}^{2}\nonumber
	\end{align}
	and
	\begin{align}\label{cosHeta}
		\left\langle 4\cos\theta H\eta_{n}, \partial_{\theta}u_{n}\right\rangle \lesssim\ \|\eta_{n}\|_{L^{2}}\| u_{n}\|_{L^{2}} \lesssim\ \| u_{n}\|_{H^{1}}^{2}.
	\end{align}
	Combining \ref{sineta}-\ref{cosHeta}, we obtain
	\begin{align*}
		\frac{d}{dt}\|\partial_{\theta}u_{n}(t)\|_{L^{2}} \leq C\|u_{n}(t)\|_{H^{1}},
	\end{align*}
	and thus
	\begin{align}\label{unH1}
		\mathop{\sup}_{0\leq t\leq T}\|u_{n}(t)\|_{H^{1}} \leq\ C,
	\end{align}
	which gives the $H^{1}$ estimate of $u_{n}$, where $C$ is a constant depending on $\mathbb{T}, u_{0}$ and $T>0$.\\
	\\
	{ \it \underline{$L^{2}$ estimate of $\partial_{t}u_{n}$ and $\partial_{t}^{2}u_{n}.$}} We multiply equation \ref{un} by $\frac{d}{dt}\tilde{\eta}_{k}^{(o)}(t)$, sum up over $k=1,2,\cdots,n$, to find
	\begin{align}
		\|\partial_{t}u_{n}\|_{L^{2}} = -\left\langle L_{1}(u_{n}), \partial_{t}u_{n}\right\rangle.
	\end{align}
	Due to $\eta_{n}(0)=0$ and $\int_{\mathbb{T}}\eta_{n}=0$, it then follows from H\"{o}lder's inequality and Poincar\'{e}'s inequality that
	\begin{align}\label{dtun}
		\mathop{\sup}_{0\leq t\leq T}\|\partial_{t}u_{n}(t)\|_{L^{2}} \leq\ C\| u_{n}(t)\|_{H^{1}} \leq\ C(\mathbb{T}, u_{0}, T),
	\end{align}
	which yields the uniform estimate of $\|\partial_{t}u_{n}(t)\|_{L^{2}}.$\\
	\\
	{\it  \underline{$L^{2}$ estimate of $\partial_{t}^{2}u_{n}.$}}	We apply $\partial_{t}$ to equation \ref{un} and multiply $\frac{d^{2}}{dt^{2}}\tilde{\eta}_{k}^{(o)}(t)$ sum up over $k=1,2,\cdots,n$ to obtain
	\begin{align}\label{pttu}
		\|\partial_{t}^{2}u_{n}\|_{L^{2}}&= -\left\langle\partial_{t}L_{1}(u_{n}), \partial_{t}^{2}u_{n}\right\rangle\nonumber\\
		&= -\big\langle \frac{1}{2}\sin2\theta\partial_{t}\partial_{\theta}u_{n}, \partial_{t}^{2}u_{n}\big\rangle+\left\langle 2\cos\theta\partial_{t}\eta_{n},\partial_{t}^{2}u_{n}\right\rangle +\left\langle 4\cos\theta\partial_{t}v_{n}, \partial_{t}^{2}u_{n}\right\rangle\\
		&\quad -\left\langle\cos^{2}\theta\partial_{t}u_{n}+2\cos\theta H(\sin\theta\partial_{t}u_{n}), \partial_{t}^{2}u_{n}\right\rangle.\nonumber
	\end{align}
	To get the estimate of $\left\langle\sin2\theta\partial_{t}\partial_{\theta}u_{n}, \partial_{t}^{2}u_{n}\right\rangle$, we first give the $H^{2}$ estimate of $u_{n}$. More generally, here we give the $H^{m}$ estimate of $u_{n}$.
	Multiplying \ref{un} by $(k+1)^{2m}\tilde{\eta}_{k}^{(o)}(t)$, summing up over $k=1,2\cdots,n$, and using the integration by parts, we obtain
	\begin{align}\label{pthemun}
		\frac{1}{2}\frac{d}{dt}\|\partial_{\theta}^{m}u_{n}\|_{L^{2}} = -\left\langle\partial_{\theta}^{m}L_{1}(u_{n}), \partial_{\theta}^{m}u_{n}\right\rangle.
	\end{align}
	For the sake of convenience, we denote the terms as lower order terms $(l.o.t)$ if their $L^{2}-$norms are bounded by $\| u_{n}\|_{H^{m-1}}$. Thus,
	\begin{align*}
		\partial_{\theta}^{m}L_{1}(u_{n}) = \frac{1}{2}\sin2\theta\partial_{\theta}^{m+1}u_{n} + (m\cos2\theta + \cos^{2}\theta)\partial_{\theta}^{m}u_{n} + 2\cos\theta H(\sin\theta\partial_{\theta}^{m}u_{n}) + l.o.t.
	\end{align*}
	It then follows from H\"{o}lder's inequality and Lemma \ref{HLp} that
	\begin{align}\label{ecosH}
		\left\langle 2\cos\theta H(\sin\theta\partial_{\theta}^{m}u_{n}), \partial_{\theta}^{m}u_{n}\right\rangle &\leq 2\| H(\sin\theta\partial_{\theta}^{m}u_{n})\|_{L^{2}}\|\partial_{\theta}^{m}u_{n}\|_{L^{2}}\nonumber\\ &\leq C\|\partial_{\theta}^{m}u_{n}\|_{L^{2}}^2.
	\end{align}
	Substituting \ref{ecosH} into \ref{pthemun}, we conclude that
	\begin{align}\label{epthemu}
		\mathop{\sup}_{0\leq t\leq T}\|u_{n}(t)\|_{H^{m}} \leq C(\mathbb{T},m,u_{0}, T).
	\end{align}
	\\
	We now turn to the estimate of $\partial_{t}\partial_{\theta}u_{n}$. Multiply equation \ref{un} by $(k+1)^{2}\frac{d}{dt}\tilde{\eta}_{k}^{(o)}(t)$, summing up over $k=1,2,\cdots,n$, and use the integration by parts to obtain
	\begin{align}\label{eptthun}
		\|\partial_{t}\partial_{\theta}u_{n}\|_{L^{2}}^{2} = -\left\langle\partial_{\theta}L_{1}(u_{n}), \partial_{t}\partial_{\theta}u_{n}\right\rangle.
	\end{align}
	It follows from H\"{o}lder's inequality that
	\begin{align}\label{ptheL}
		\left\langle\partial_{\theta}L_{1}(u_{n}), \partial_{t}\partial_{\theta}u_{n}\right\rangle &\leq\ \|\partial_{\theta}L_{1}(u_{n})\|_{L^{2}}\|\partial_{t}\partial_{\theta}u_{n}\|_{L^{2}}\nonumber\\
		& \lesssim\| u_{n}\|_{H^{2}}\|\partial_{t}\partial_{\theta}u_{n}\|_{L^{2}}.
	\end{align}
	Substituting \ref{ptheL} into \ref{eptthun}, we obtain
	\begin{align}\label{epttheu}
		\mathop{\sup}_{0\leq t\leq T}\|\partial_{t}\partial_{\theta}u_{n}(t)\|_{L^{2}}\leq\ C(\mathbb{T}, \| u_{0}\|_{H^{2}}, T).
	\end{align}
	We now  estimate  $\partial_{t}^{2}u_{n}$.
	Due to $\partial_{t}\eta_{n}(0)=0$, we have
	\begin{align}\label{cosptetan}
		\left\langle\cos\theta\partial_{t}\eta_{n}, \partial_{t}^{2}u_{n}\right\rangle &\leq\ \|\partial_{t}\eta_{n}\|_{L^{\infty}}\|\cos\theta\|_{L^{2}}\|\partial_{t}^{2}u_{n}\|_{L^{2}}\nonumber\\
		&\lesssim\|\partial_{t}\partial_{\theta}\eta_{n}\|_{L^{2}}\|\partial_{t}^{2}u_{n}\|_{L^2}\\
		&\lesssim\| u_{n}\|_{H^{2}}\|\partial_{t}^{2}u_{n}\|_{L^{2}}.\nonumber
	\end{align}
	For the term $\left\langle\cos\theta\partial_{t}v_{n}, \partial_{t}^{2}u_{n}\right\rangle$, it follows from H\"{o}lder's inequality, Sobolev's embedding, and Lemma \ref{le hilbert estimate} that
	\begin{align}\label{cosptvn}
		\left\langle\cos\theta\partial_{t}v_{n}, \partial_{t}^{2}u_{n}\right\rangle &\leq\ \|\partial_{t}v_{n}\|_{L^{\infty}}\|\cos\theta\|_{L^{2}}\|\partial_{t}^{2}u_{n}\|_{L^{2}}\nonumber\\
		&\lesssim\| H(\partial_{t}\eta_{n})\|_{L^{2}}\|\partial_{t}^{2}u_{n}\|_{L^{2}}\nonumber\\
		&\lesssim\|\partial_{t}\eta_{n}\|_{L^{2}}\|\partial_{t}^{2}u_{n}\|_{L^{2}}\\
		&\lesssim\|\partial_{t}\partial_{\theta}\eta_{n}\|_{L^{1}}\|\partial_{t}^{2}u_{n}\|_{L^{2}}\nonumber\\
		&\lesssim\|\partial_{t}u_{n}\|_{L^{2}}\|\partial_{t}^{2}u_{n}\|_{L^{2}}.\nonumber
	\end{align}
	Substituting \ref{epthemu} and \ref{epttheu}-\ref{cosptvn} into \ref{pttu}, we obtain
	\begin{align}\label{epttu}
		\mathop{\sup}_{0\leq t\leq T}\|\partial_{t}^{2}u_{n}(t)\|_{L^{2}}\leq\ C(\mathbb{T}, \| u_{0}\|_{H^{2}}, T),
	\end{align}
	which yields the uniform estimate of $\partial_{t}^{2}u_{n}$.\\
	\\
	{\it \underline{Higher order estimates of $\partial_{t}u_{n}$ and $\partial_{t}^{2}u_{n}.$}}	In order to obtain the $H^{m-1}$ estimate of $\partial_{t}u_{n}$, we multiply equation \ref{un} by $(k+1)^{2m-2}\frac{d}{dt}\tilde{\eta}_{k}^{(o)}(t)$ and sum up
	over $k=1,2,\cdots,n.$ It then follows from the integration by parts that
	\begin{align}\label{ptthu}
		\|\partial_{t}\partial_{\theta}^{m-1}u_{n}\|_{L^{2}} = -\left\langle-\partial_{\theta}^{m-1}L_{1}(u_{n}), \partial_{t}\partial_{\theta}^{m-1}u_{n}\right\rangle.
	\end{align}
	For the $\partial_{\theta}^{m-1}L_{1}(u_{n})$, we have
	\begin{align*}
		\partial_{\theta}^{m-1}L_{1}(u_{n}) = \frac{1}{2}\sin2\theta\partial_{\theta}^{m}u_{n} + l.o.t,
	\end{align*}
	where $(l.o.t)$ denotes the terms that their $L^{2}-$norms are bounded by $\| u_{n}\|_{H^{m-1}}$.
	Then we obtain
	\begin{align*}
		\|\partial_{t}\partial_{\theta}^{m-1}u_{n}\|_{L^{2}}\leq\ \| u_{n}\|_{H^{m}}.
	\end{align*}
	Thus we can conclude that
	\begin{align}\label{eptm1u}
		\mathop{\sup}_{0\leq t\leq T}\|\partial_{t}u_{n}(t)\|_{H^{m-1}} \leq\ C(\mathbb{T}, \| u_{0}\|_{H^{m}}, T),
	\end{align}
	where we have used Sobolev's embedding and H\"{o}lder's inequality.\\
	
	Similarly, to get the $H^{m-2}$ estimate of $\partial_{t}^{2}u_{n}$. We apply $\partial_{t}$ to equation \ref{un} and multiply $(k+1)^{2m-4}\frac{d^{2}}{dt^{2}}\tilde{\eta}_{k}^{(o)}(t)$, sum up over $k=1.2\cdots,n$ and deduce that
	\begin{align}\label{pttm2u}
		\|\partial_{t}^{2}\partial_{\theta}^{m-2}u_{n}\|_{L^{2}}^{2} = -\left\langle\partial_{t}^{2}\partial_{\theta}^{m-2}L_{1}(u_{n}), \partial_{t}\partial_{\theta}^{m-2}u_{n}\right\rangle
		.
	\end{align}
	Substituting \ref{epthemu} and \ref{eptm1u} into \ref{pttm2u}, we obtain
	\begin{align}\label{epttm2u}
		\|\partial_{t}u_{n}(t)\|_{H^{m-2}} \leq\ C(\mathbb{T}, m, u_{0}, T),
	\end{align}
	which yields the uniform estimate of $\|\partial_{t}^{2}u_{n}(t)\|_{H^{m-2}}$.\\
	
	Due to $u_{n} = -\sqrt{\pi}\rho^{1/2}\partial_{\theta}\eta_{n}$, \ref{epthemu},\ref{eptm1u} and \ref{epttm2u} can be rewritten as
	\begin{align}
		&\mathop{\sup}_{0\leq t\leq T}\|\rho^{1/2}\partial_{\theta}\eta_{n}(t)\|_{H^{m}}\leq C(\mathbb{T},m,\rho^{1/2}\partial_{\theta}\eta_{0},T),\label{etam}\\
		&\mathop{\sup}_{0\leq t\leq T}\|\rho^{1/2}\partial_{t}\partial_{\theta}\eta_{n}(t)\|_{H^{m-1}}\leq C(\mathbb{T},m,\rho^{1/2}\partial_{\theta}\eta_{0},T),\label{tetam}\\
		&\ \mathop{\sup}_{0\leq t\leq T}\|\rho^{1/2}\partial_{t}^{2}\partial_{\theta}\eta_{n}(t)\|_{H^{m-2}}\leq\ C(\mathbb{T},m,\rho^{1/2}\partial_{\theta}\eta_{0},T)\label{t2etam}
	\end{align}
	for any $T>0$.
	Meanwhile, it follows from \ref{etam}, \ref{tetam} and \ref{t2etam} that
	\begin{corollary}\label{coro}
		Assume that the conditions of Theorem \ref{the existence} hold. Then there exists a constant $C(\mathbb{T},m,\eta_{0},T)$ such that the coefficients of $\eta_{n}$ in \ref{etans} satisfy
		\begin{align}
			&\mathop{\sup}_{0\leq t\leq T}|\tilde{\eta}_{k}^{(o)}(t)|\ \leq\ \frac{C(\mathbb{T},m,\eta_{0},T)}{k^{m}},\label{etak}\\
			&\mathop{\sup}_{0\leq t\leq T}|\frac{d}{dt}\tilde{\eta}_{k}^{(o)}(t)|\ \leq\ \frac{C(\mathbb{T},m,\eta_{0},T)}{k^{m-1}},\label{dtetak}\\
			&\mathop{\sup}_{0\leq t\leq T}|\frac{d^{2}}{dt^{2}}\tilde{\eta}_{k}^{(o)}(t)|\ \leq\ \frac{C(\mathbb{T},m,\eta_{0},T)}{k^{m-2}}.\label{dttetak}
		\end{align}
	\end{corollary}
	
	{\bf Step 3. Convergence of the approximate solutions.}
	
	Combining \ref{epthemu}, \ref{eptm1u} and \ref{epttm2u}, there exists a subsequence $\{u_{n_{l}}\}_{n_{l}=1}^{\infty}\subset\{u_{n}\}_{n=1}^{\infty}$ such that $u_{n_{l}}$ converges to $u$ weakly in $L^{2}([0,T];H^{m}(\mathbb{T}))$, $\partial_{t}u_{n_{l}}$ converges to $\partial_{t}u$ weakly in $L^{2}([0,T];H^{m-1}(\mathbb{T}))$ and $\partial_{t}^{2}u_{n_{l}}$ converges to $\partial_{t}^{2}u$ weakly in $L^{2}([0,T];H^{m-2}(\mathbb{T}))$. For simplicity, we set $n_{l}=n.$
	Fix a positive integer $N$ and choose a function $\phi\in C^{1}([0,T];C^{2}(\mathbb{T}))$ such that
	\begin{align}\label{phi}
		\phi = \sum_{k=1}^{N}a_{k}^{(o)}(t)\sin k\theta,
	\end{align}
	where $a_{k}^{(o)}(t)(k=1,2,\cdots,N)$ is a smooth function. And we choose $n\geq N,$ multiply \ref{un} by $a_{k}^{(o)}(t)$, sum up over $k=0,1,2,\cdots,n,$ and integrate with respect to $t$, it is easy to get that
	\begin{align}\label{intun}
		\int_{0}^{T}\left\langle\partial_{t}u_{n}(t),\phi(t)\right\rangle + \left\langle L_{1}(u_{n})(t), \phi(t)\right\rangle dt = 0.
	\end{align}
	Since $\{u_{n}\}_{n=1}^{\infty}$ converges to $u$ weakly in $L^{2}([0,T];H^{m}(\mathbb{T}))$ and $\{\partial_{t}u_{n}\}_{n=1}^{\infty}$ converges to $\partial_{t}u$ weakly in $L^{2}([0,T];H^{m-1}(\mathbb{T}))$, it can be concluded that
	\begin{align}\label{intphi}
		\int_{0}^{T}\left\langle\partial_{t}u(t),\phi(t)\right\rangle + \left\langle L_{1}(u)(t), \phi(t)\right\rangle dt = 0.
	\end{align}
	And we can obtain \ref{intphi} for all functions $\phi\in C^{1}([0,T];C^{2}(\mathbb{T}))$, because functions of the form \ref{phi} are dense in this space.
	Combining estimates \ref{etak}-\ref{dttetak} of coefficients of $\eta_{n}$, we deduce that $\{u_{n}\}_{n=1}^{\infty}$ converges to $u=-\sqrt{\pi}\rho^{1/2}\partial_{\theta}\eta$ strongly in $C(0,T;H^{m}(\mathbb{T}))$.
	
	{\bf Step 4. Uniqueness.}
	
	If both $\eta_{1}$ and $\eta_{2}$ are solutions to \ref{etav}, it can be obtained that
	\begin{align}
		\|\rho^{1/2}\partial_{\theta}\eta_{1}-\rho^{1/2}\partial_{\theta}\eta_{2}\|_{L^{2}} \leq\ e^{CT}\|\rho^{1/2}\partial_{\theta}\eta_{0}-\rho^{1/2}\partial_{\theta}\eta_{0}\|_{L^{2}} =0,
	\end{align}
	which means that the solution to \ref{etav} is unique, and the proof of Theorem \ref{the existence} is finished.
	
	\section{Instability of  the linearized equation around $-\sin 2\theta$}\label{sec linear instability}
	In this section, we give the proof of Theorems \ref{the linear instability} and Corollary \ref{coro linear instability}.  We first demonstrate that the coefficients of the solution $\eta = \sum_{k\geq 1}\tilde{\eta}_{k}(t)\tilde{e}_{k}^{(o)}$ obtained in Theorem \ref{the existence} can be expressed as an infinite-dimensional ordinary system \ref{dteta}. Then we  derive a second-order differential inequality \ref{sodi} concerning $\|\eta\|_{\mathcal{H}_{DW}}$. Finally we prove that $\|\eta\|_{\mathcal{H}_{DW}}\ (\eta\neq 0)$ will grow exponentially with time under appropriate initial data, which implies the instability of the linearized equation \ref{etav}.
	
	Before proving Theorem \ref{the linear instability}, we state a lemma that will be used later. Its proof is deferred to the Appendix.
	\begin{lemma}\label{lepsk}
		For any positive integer $k$, the matrix
		\begin{equation}\label{matrixAk}
			A_{k}=\left(\begin{array}{lc}
				a_{k} & \varepsilon_{k} \\
				\varepsilon_{k} & a_{k+2}
			\end{array}\right)
		\end{equation}
		is a positive definite matrix,
		where
		\begin{align}
			a_{k} = \left(-d_{k+2} + d_{k}\right)^{2} = \left(-\frac{1}{2} - \frac{2k^{2}-4k-8}{(k+2)^{2}k^{2}}\right)^2
		\end{align}
		and
		\begin{align}\label{epsk}
			\varepsilon_{k} = -2d_{k+2}^{2}+d_{k}d_{k+2}+ d_{k+2}d_{k+4}
			=\frac{-2k^{3}+32k+32}{(k+2)^{4}(k+4)}.
		\end{align}
		
		Furthermore, there exist two absolute constants $\lambda_{1},\lambda_{2} > 0$ such that the eigenvalues $\lambda_{k}^{1},\ \lambda_{k}^2$ of $A_{k}$ and coefficients $a_{k}$ satisfying
		\begin{align}
			0 < \lambda_{1} < \lambda_{inf}\leq a_{k},\,\lambda_{k}^{1},\, \lambda_{k}^{2} \leq \lambda_{\sup} < \lambda_{2}
		\end{align}
		for $k\geq 1$, where
		\begin{align}
			\lambda_{k}^{1}=\frac{a_{k}+a_{k+2}-\sqrt{(a_{k}-a_{k+2})^{2}+4\varepsilon_{k}^{2}}}{2},
		\end{align}
		\begin{align}
			\lambda_{k}^{2}=\frac{a_{k}+a_{k+2}+\sqrt{(a_{k}-a_{k+2})^{2}+4\varepsilon_{k}^{2}}}{2}
		\end{align}
		and
		\begin{align*}
			\lambda_{\sup} = \mathop{\sup}_{k\geq1}\{a_{k}, \lambda_{k}^{1}, \lambda_{k}^{2} \},\ \lambda_{inf} = \mathop{\inf}_{k\geq1}\{a_{k}, \lambda_{k}^{1}, \lambda_{k}^{2}\}.
		\end{align*}
	\end{lemma}
	\begin{remark}
		It is easy to prove that $\lambda_{k}^{1},\lambda_{k}^{2}$ and $a_{k}$ converge to $\frac{1}{4}$ as $k\rightarrow\infty$. However, the proof of Lemma \ref{lepsk} involves straightforward but tedious calculations using the explicit expressions for $\lambda_{k}^{1},\lambda_{k}^{2}$ and $a_{k}$. We provide a rigorous proof in the Appendix. Also, the numerical illustrations of $a_{k},\lambda_{k}^{1},\lambda_{k}^{2}$ are presented in Appendix.
	\end{remark}

	We now proceed to the proof of Theorem \ref{the linear instability}.
	\begin{proof}[Proof of Theorem \ref{the linear instability}]  Consider the solution obtained in Theorem \ref{the existence}, which can be expressed as  $\eta = \sum_{k\geq 1}\tilde{\eta}_{k}^{(o)}(t)\tilde{e}_{k}^{(o)}$. Note that although the coefficients $\tilde{\eta}_{k}^{(o)}(t)$ here is possibly not same as in \ref{etans}, the decay properties in Corollary \ref{coro} hold true for both of them due to the convergence of the approximate solutions and the uniqueness of the solutions.
		
		{\bf Step 1. The infinite dimensional ODE system.}
		
		Direct computations yield
		\begin{align}
			-Le_{k}^{(o)} = A_{k}e_{k+2}^{(o)} + B_{k}e_{k-2}^{(o)}, \quad k > 2,
		\end{align}
		where the linear operator $L$ is defined as in \ref{operator} and
		\begin{align}
			A_{k}=-\frac{(k-2)^{2}}{4k},\quad B_{k}=
			\frac{(k+2)(k-2)}{4k},\quad k > 2.
		\end{align}
		For $k=1$, one has
		\begin{align*}
			-Le_{1}^{(o)}=-\frac{1}{4} e_{3}^{(o)} +\frac{3}{4} e_{1}^{(o)}.
		\end{align*}
		For $k=2$, one has
		\begin{align*}
			-L{e_{2}^{(o)}} = 0.
		\end{align*}
		It follows that
		\begin{align*}
			-L\tilde{e}_{k}^{(o)} & = \frac{A_{k+2}}{k+2}e_{k+4}^{(o)}+\frac{B_{k+2}}{k+2}e_{k}^{(o)} - \frac{A_{k}}{k}e_{k+2}^{(o)} -\frac{B_{k}}{k}e_{k-2}^{(o)}\\
			&=-\frac{k^2}{4(k+2)^2}e_{k+4}^{(o)} + \frac{(k+4)k}{4(k+2)^2}e_k^{(o)}+\frac{(k-2)^2}{4k^2}e_{k+2}^{(o)} -\frac{(k+2)(k-2)}{4k^2}e_{k-2}^{(o)}\\
			&=-\frac{k^{2}(k+4)}{4(k+2)^{2}}\left(\frac{e_{k+4}^{(o)}}{k+4}-\frac{e_{k+2}^{(o)}}{k+2}\right) + \left[\frac{(k-2)^{2}(k+2)}{4k^{2}}-\frac{k^{2}(k+4)}{4(k+2)^{2}}\right]\left(\frac{e_{k+2}^{(o)}}{k+2}-\frac{e_{k}^{(o)}}{k}\right) \\
			&\quad + \frac{(k-2)^{2}(k+2)}{4k^{2}}\left(\frac{e_{k}^{(o)}}{k}-\frac{e_{k-2}^{(o)}}{k-2}\right)\\
			&=-\frac{k^{2}(k+4)}{4(k+2)^{2}}\tilde{e}_{k+2}^{(o)} + \left[\frac{(k-2)^{2}(k+2)}{4k^{2}}-\frac{k^{2}(k+4)}{4(k+2)^{2}}\right]\tilde{e}_{k}^{(o)} + \frac{(k-2)^{2}(k+2)}{4k^{2}}\tilde{e}_{k-2}^{(o)},
		\end{align*}
		which is
		\begin{align}\label{equ4.3}
			-L\tilde{e}_{k}^{(o)} = -d_{k+2}\tilde{e}_{k+2}^{(o)} + (-d_{k+2}+d_{k})\tilde{e}_{k}^{(o)} + d_{k}\tilde{e}_{k-2}^{(o)},
		\end{align}
		where
		\begin{align*}
			d_{k} = \frac{(k-2)^{2}(k+2)}{4k^{2}}
		\end{align*}
		and
		\begin{align}
			-d_{k+2} + d_{k} & = -\frac{k^{2}(k+4)}{4(k+2)^{2}} + \frac{(k-2)^{2}(k+2)}{4k^{2}}\nonumber\\
			& = -\frac{k^{4}+4k^{3}+8k^{2}-8k-16}{2(k+2)^{2}k^{2}}\\
			& = -\frac{1}{2} - \frac{2k^{2}-4k-8}{(k+2)^{2}k^{2}}.\nonumber
		\end{align}
		Note that $d_{2}=0$. The above equality holds true for all $k\geq 2.$ For $k=1$, we have
		\begin{align}
			-L \tilde{e}_{1}^{(o)}=-d_{3} \tilde{e}_{3}^{(o)}+\left(-d_{3}+d_{1}\right) \tilde{e}_{1}^{(o)}.
		\end{align}
		
		In view of Corollary \ref{coro}, the equation \ref{etav} can be expressed as the following infinite-dimensional ordinary differential equation (ODE) system
		\begin{align}\label{dteta}
			\frac{d}{dt}\tilde{\eta}_{k}^{(o)}(t) = -d_{k}\tilde{\eta}_{k-2}^{(o)}(t) + (d_{k}-d_{k+2})\tilde{\eta}_{k}^{(o)}(t) + d_{k+2}\tilde{\eta}_{k+2}^{(o)}(t),\quad k\geq 1.
		\end{align}
		Here, $\tilde{\eta}_{-1}^{(o)}(t)$  and $\tilde{\eta}_{0}^{(o)}(t)$ are understood to be 0.\\
		
		{\bf Step 2. The quadratic form on $\tilde{\eta}_{k}^{(o)}(t)$ and $\tilde{\eta}_{k+2}^{(o)}(t)$.}
		
		In this step, we demonstrate that the non-trivial solution $\|\eta\|_{\mathcal{H}_{DW}}$ obtained in Theorem \ref{the existence} satisfies a second-order ordinary differential inequality.
		
		Applying the estimate of $\tilde{\eta}_{k}^{(o)}(t)$ from Corollary \ref{coro} yields
		\begin{align}
			|\frac{d}{dt}\left(d_{k}-d_{k+2}\right)\left(\tilde{\eta}_{k}^{(o)}(t)\right)^{2}|\leq \frac{C}{k^{2m-1}}, \quad
		\end{align}
		where $0\leq t < \infty$, $k\geq 1$ and $m>3$ an integer. It follows that
		\begin{align}
			\frac{d}{dt}\left\langle-L\eta,\eta\right\rangle_{\rho} &=\frac{d}{dt}\sum_{k\geq 1}\left(d_{k}-d_{k+2}\right)\left(\tilde{\eta}_{k}^{(o)}(t)\right)^{2}\nonumber\\
			&=\sum_{k\geq 1}\frac{d}{dt}\left(d_{k}-d_{k+2}\right)\left(\tilde{\eta}_{k}^{(o)}(t)\right)^{2}.\nonumber
		\end{align}

		We multiply equation \ref{dteta} by $(-d_{k+2}+d_{k})\tilde{\eta}_{k}^{(o)}(t)$ to obtain
		\begin{align}\label{dteta2}
			\frac{1}{2}\frac{d}{dt}(d_{k}-d_{k+2})\left(\tilde{\eta}_{k}^{(o)}(t)\right)^{2} & = -d_{k}(d_{k}-d_{k+2})\tilde{\eta}_{k-2}^{(o)}(t)\tilde{\eta}_{k}^{(o)}(t) + (d_{k}-d_{k+2})^{2}\left(\tilde{\eta}_{k}^{(o)}(t)\right)^{2}\nonumber\\
			&\quad +d_{k+2}(d_{k}-d_{k+2})\tilde{\eta}_{k}^{(o)}(t)\tilde{\eta}_{k+2}^{(o)}(t)
		\end{align}
		for $k\geq 1.$
		
		Considering the sum $S_{n}=\sum_{k=1}^{n}\frac{d}{dt}\left(d_{k}-d_{k+2}\right)\left(\tilde{\eta}_{k}^{(o)}(t)\right)^{2}$ and combing \ref{dteta2}, we obtain
		\begin{align}\label{dtL}
			S_{n}=&2\sum_{k=1}^{n}\left[-d_{k}(d_{k}-d_{k+2})\tilde{\eta}_{k-2}^{(o)}\tilde{\eta}_{k}^{(o)} + (d_{k}-d_{k+2})^{2}\left(\tilde{\eta}_{k}^{(o)}\right)^{2} +d_{k+2}(d_{k}-d_{k+2})\tilde{\eta}_{k}^{(o)}\tilde{\eta}_{k+2}^{(o)}\right]\nonumber\\
			=& \left(-d_{3} + d_{1}\right)^{2}\left(\tilde{\eta}_{1}^{(o)}\right)^{2}\nonumber\\
			& + \left(-d_{4} + d_{2}\right)^{2}\left(\tilde{\eta}_{2}^{(o)}\right)^{2}\nonumber\\
			& +\left(-d_{3} + d_{1}\right)^{2}\left(\tilde{\eta}_{1}^{(o)}\right)^{2} + 2\left(-2d_{3}^{2} + d_{1}d_{3} +d_{3}d_{5}\right)\tilde{\eta}_{1}^{(o)}\tilde{\eta}_{3}^{(o)} + \left(-d_{5} + d_{3}\right)^{2}\left(\tilde{\eta}_{3}^{(o)}\right)^{2}\nonumber\\
			& +\left(-d_{4} + d_{2}\right)^{2}\left(\tilde{\eta}_{2}^{(o)}\right)^{2} + 2\left(-2d_{4}^{2} + d_{2}d_{4} +d_{4}d_{6}\right)\tilde{\eta}_{2}^{(o)}\tilde{\eta}_{4}^{(o)} + \left(-d_{6} + d_{4}\right)^{2}\left(\tilde{\eta}_{6}^{(o)}\right)^{2}\nonumber\\
			&\cdots\\
			& +\left(-d_{n-1} + d_{n-3}\right)^{2}\left(\tilde{\eta}_{n-3}^{(o)}\right)^{2} + 2\left(-2d_{n-1}^{2} + d_{n-3}d_{n-1} +d_{n-1}d_{n+1}\right)\tilde{\eta}_{n-3}^{(o)}\tilde{\eta}_{n-1}^{(o)}\nonumber\\
			& + \left(-d_{n+1} + d_{n-1}\right)^{2}\left(\tilde{\eta}_{n-1}^{(o)}\right)^{2}\nonumber\\
			& +\left(-d_{n+1} + d_{n-1}\right)^{2}\left(\tilde{\eta}_{n-1}^{(o)}\right)^{2} + 2d_{n+1}\left(-d_{n+1}+d_{n-1}\right)\tilde{\eta}_{n-1}^{(o)}\tilde{\eta}_{n+1}^{(o)}\nonumber\\
			&+\left(-d_{n-2} + d_{n}\right)^{2}\left(\tilde{\eta}_{n-2}^{(o)}\right)^{2} + 2\left(-2d_{n}^{2} + d_{n-2}d_{n} +d_{n}d_{n+2}\right)\tilde{\eta}_{n-2}^{(o)}\tilde{\eta}_{n}^{(o)} + \left(-d_{n+2} + d_{n}\right)^{2}\left(\tilde{\eta}_{n}^{(o)}\right)^{2}\nonumber\\
			& +\left(-d_{n+2} + d_{n}\right)^{2}\left(\tilde{\eta}_{n}^{(o)}\right)^{2} + 2d_{n+2}\left(-d_{n+2}+d_{n}\right)\tilde{\eta}_{n}^{(o)}\tilde{\eta}_{n+2}^{(o)}\nonumber\\
			:= & \left(-d_{3} + d_{1}\right)^{2}\left(\tilde{\eta}_{1}^{(o)}\right)^{2}  + \left(-d_{4} + d_{2}\right)^{2}\left(\tilde{\eta}_{2}^{(o)}\right)^{2} +\sum_{k=1}^{n-2}f_{k}\nonumber \\
			&+\left(-d_{n+1} + d_{n-1}\right)^{2}\left(\tilde{\eta}_{n-1}^{(o)}\right)^{2}  + \left(-d_{n+2} + d_{n}\right)^{2}\left(\tilde{\eta}_{n}^{(o)}\right)^{2}+ R_{n-1}+R_{n}.\nonumber
		\end{align}
		Here the quadratic form  $f_{k}$ is defined as
		\begin{align}\label{form}
			f_{k} = (-d_{k+2}+d_{k})^2\left(\tilde{\eta}_{k}^{(o)}\right)^{2} + 2(-2d_{k+2}^{2}+d_{k}d_{k+2}+ d_{k+2}d_{k+4})\tilde{\eta}_{k}^{(o)}\tilde{\eta}_{k+2}^{(o)}+(-d_{k+4}+d_{k+2})^2\left(\tilde{\eta}_{k+2}^{(o)}\right)^{2}
		\end{align}
		for $k=1,2,\cdots,n-2.$ The remainder term $R_{k}$ is given by
		\begin{align}\label{Rk}
			R_{k}(t) = 2d_{k+2}(-d_{k+2}+d_{k})\tilde{\eta}_{k}^{(o)}(t)\tilde{\eta}_{k+2}^{(o)}(t)
		\end{align}
		for $k=n-1,n$. Applying the estimates from Corollary \ref{coro}, we obtain
		\begin{align}\label{cRK}
			|R_{k}(t)|\leq \frac{C}{k^{2m-1}},
		\end{align}
		where C is a constant depending on $m, \mathbb{T}$ and the initial data $\eta_{0}$ for $0\leq t < \infty$. Consequently, both $R_{n-1}(t)$ and $ R_{n}(t)$ tend to 0 as $n\rightarrow \infty$.
		
		Using Lemma \ref{lepsk}, we derive the following inequality
		\begin{align}\label{cfk}
			\lambda_{inf}\left[\left(\tilde{\eta}_{k}^{(o)}\right)^{2}+\left(\tilde{\eta}_{k+2}^{(o)}\right)^{2}\right]\leq f_{k}\leq \lambda_{\sup}\left[\left(\tilde{\eta}_{k}^{(o)}\right)^{2}+\left(\tilde{\eta}_{k+2}^{(o)}\right)^{2}\right]
		\end{align}
		for $k\geq 1$.
		Substituting \ref{cfk} into \ref{dtL}, yields
		\begin{align}
			S_{n}
			\leq & \left(-d_{3} + d_{1}\right)^{2}\left(\tilde\eta_{1}^{(o)}\right)^{2} + \left(-d_{4} + d_{2}\right)^{2}\left(\tilde\eta_{2}^{(o)}\right)^{2} + \sum_{k=1}^{n-2}\lambda_{\sup}\left[\left(\tilde{\eta}_{k}^{(o)}\right)^{2}+\left(\tilde{\eta}_{k+2}^{(o)}\right)^{2}\right]\nonumber\\
			&+\left(-d_{n+1} + d_{n-1}\right)^{2}\left(\tilde{\eta}_{n-1}^{(o)}\right)^{2} +\left(-d_{n+2} + d_{n}\right)^{2}\left(\tilde{\eta}_{n}^{(o)}\right)^{2} + R_{n-1} + R_{n}\\
			\leq&\ 2\lambda_{\sup}\sum_{k=1}^{n}\left(\tilde{\eta}_{k}^{(o)}\right)^{2} + R_{n-1} + R_{n}\nonumber
		\end{align}
		and
		\begin{align}
			S_{n} \geq 2\sum_{k=1}^{n}\lambda_{inf}\left(\tilde{\eta}_{k}^{(o)}\right)^{2}  + R_{n-1} + R_{n}.
		\end{align}
		
		Consequently, it follows that
		\begin{align}
			2\lambda_{inf}\sum_{k=1}^{n}\left(\tilde{\eta}_{k}^{(o)}\right)^{2} + R_{n-1} + R_{n}\leq S_{n}\leq 2\lambda_{\sup}\sum_{k=1}^{n}\left(\tilde{\eta}_{k}^{(o)}\right)^{2}+R_{n-1} + R_{n},
		\end{align}
		which leads to a second-order ordinary differential inequality concerning $\|\eta(t)\|_{\mathcal{H}_{DW}}^{2}$
		\begin{align}\label{sodi}
			4\lambda_{1}\|\eta(t)\|_{\mathcal{H}_{DW}}^{2}< \frac{d^{2}}{dt^{2}}\|\eta(t)\|_{\mathcal{H}_{DW}}^{2}<4\lambda_{2}\|\eta(t)\|_{\mathcal{H}_{DW}}^{2},
		\end{align}
		as $n\rightarrow\infty$ for $0\leq t\leq T$, where $0<\lambda_1<\lambda_2$ are two positive constants as in Lemma \ref{lepsk} and we have used the decay property of $R_{k}$ in \ref{cRK} for $k=1,2,\cdots$ and $T>0$ is any positive constant.

		{\bf Step 3. The second-ordinary differential inequality.}
		
		In this step, we will apply the differential inequality \eqref{sodi} and the comparison theorem (Lemma \ref{Comp}) to finish the proof of Theorem \ref{the linear instability}. To this end, we first consider
		\begin{equation}\label{uode}
			\left\{\begin{split}
				&u''(t) = 4\lambda_{i}u(t),\ t\geq 0,\\
				&u(0)=0,\ u'(0)=1
			\end{split}\right.
		\end{equation}
		for $i=1,2$. The solution of \ref{uode} is given by
		\begin{align}
			u(t) = \frac{1}{4\sqrt{\lambda_{i}}}e^{2\sqrt{\lambda_{i}}t} - \frac{1}{4\sqrt{\lambda_{i}}}e^{-2\sqrt{\lambda_{i}}t} > 0
		\end{align}
		for $t > 0$ with $i=1,2$.
		
		Then we solve
		\begin{equation}\label{uode2}
			\left\{\begin{split}
				&y''(t) = 4\lambda_{i}y(t),\ t\geq 0,\\
				&y(0)=\left\langle\eta_0, \eta_0\right\rangle_{\rho},\ y'(0)=2\left\langle -L \eta_0, \eta_0\right\rangle_{\rho},
			\end{split}\right.
		\end{equation}
		for $i=1,2$. The solution of \ref{uode2} is given by
		\begin{align*}
			y(t)=J_{i}(t):= \frac{\left\langle\eta_0, \eta_0\right\rangle_{\rho} + \frac{1}{\sqrt{ \lambda_i}}\left\langle -L \eta_0, \eta_0\right\rangle_{\rho}}{2} e^{2 \sqrt{ \lambda_{i} }t} + \frac{\left\langle\eta_0, \eta_0\right\rangle_{\rho} - \frac{1}{\sqrt{ \lambda_i}}\left\langle -L \eta_0, \eta_0\right\rangle_{\rho}}{2} e^{-{2 \sqrt{ \lambda_{i} }t}}
		\end{align*}
		for $i=1,2$, respectively.

		Clearly, for $t>0$, we have
		\begin{align*}
			J_{i}(t) = \frac{1}{2}\left\langle\eta_{0}, \eta_{0}\right\rangle_{\rho}\cosh(2\sqrt{\lambda_{i}}t) + \frac{1}{2\sqrt{\lambda_{i}}}\left\langle-L \eta_{0}, \eta_{0}\right\rangle_{\rho}\sinh (2\sqrt{\lambda_{i}}t) > 0,
		\end{align*}
		provided that $\left\langle-L \eta_0, \eta_0\right\rangle_{\rho} \geq -\sqrt{\lambda_{i}}\left\langle \eta_{0}, \eta_{0}\right\rangle_{\rho}$ for $i=1,2$ and $\eta_{0}\neq 0$.
		
		Observe that
		\begin{align*}
			\cosh x=\frac{e^x+e^{-x}}{2},\ \sinh x=\frac{e^x-e^{-x}}{2}
		\end{align*}	
		and that both $\cosh x$ and $(\sinh x)/x$ are strictly increasing functions for $x>0$. Therefore, it follows that
		\begin{align*}
			J_{2}(t)-J_{1}(t) =& \frac{1}{2}\left\langle \eta_{0}, \eta_{0}\right\rangle_{\rho}\cosh (2\sqrt{\lambda_{2}}t) + \frac{1}{2\sqrt{\lambda_{2}}}\left\langle-L \eta_0, \eta_0\right\rangle_{\rho}\sinh (2\sqrt{\lambda_{2}}t)\\
			& -\frac{1}{2}\left\langle \eta_{0}, \eta_{0}\right\rangle_{\rho}\cosh (2\sqrt{\lambda_{1}}t) - \frac{1}{2\sqrt{\lambda_{1}}}\left\langle-L \eta_0, \eta_0\right\rangle_{\rho}\sinh (2\sqrt{\lambda_{1}}t)\\
			=&\frac{1}{2}\left\langle \eta_{0}, \eta_{0}\right\rangle_{\rho}\left(\cosh (2\sqrt{\lambda_{2}}t) - \cosh (2\sqrt{\lambda_{1}}t)\right)\\
			&+ \frac{1}{2}\left\langle-L \eta_0, \eta_0\right\rangle_{\rho}\left(\frac{1}{\sqrt{\lambda_{2}}}\sinh (2\sqrt{\lambda_{2}}t) - \frac{1}{\sqrt{\lambda_{1}}}\sinh (2\sqrt{\lambda_{1}}t)\right) > 0
		\end{align*}
		for $t> 0$, provided that $\left\langle-L \eta_0, \eta_0\right\rangle_{\rho}\geq 0\ (\eta_{0}\neq 0)$.
		
		Applying Lemma \ref{Comp}, we obtain
		\begin{align}\label{etaine}
			J_{1}(t) < \|\eta(t)\|_{\mathcal{H}_{DW}}^{2} < J_{2}(t)
		\end{align}
		for $t>0$.
		The proof of Theorem \ref{the linear instability} is finished.
	\end{proof}
	
	Next, we give the proof of Corollary \ref{coro linear instability}.
	\begin{proof}[Proof of Corollary \ref{coro linear instability}]
		It suffices to prove that the initial data presented in (i) and (ii) satisfy the conditions in Theorem  \ref{the linear instability}.
		
		If $\eta_{0} = a_{1}\tilde{e}_{1}^{(o)}(a_{1}\neq 0)$, then $\eta_{0}\in\mathcal{H}_{DW}$ and $\rho^{1/2}\partial_{\theta}\eta_{0}\in H^{m}$ with $m>3$. Moreover, it yields
		\begin{align}
			\left\langle-L\eta_{0},\eta_{0}\right\rangle_{\rho} &= (-d_{3} + d_{1})a_{1}^{2}\nonumber\\
			&=\frac{11}{18} a_{1}^{2} > 0.
		\end{align}
		
		If $\eta_{0} = a_{1}\tilde{e}_{1}^{(o)} + a_{k}\tilde{e}_{k}^{(o)}$ with $a_{1}, a_{k}\neq 0$ for $k\geq 2$, then direct calculations show that  $\eta_{0}\in\mathcal{H}_{DW}$ and $\rho^{1/2}\partial_{\theta}\eta_{0}\in H^{m}$ with $m>3$. Moreover, it yields
		\begin{align}
			\left\langle-L\eta_{0},\eta_{0}\right\rangle_{\rho}
			& = (-d_{3} + d_{1})a_{1}^{2} + (-d_{k+2} + d_{k})a_{k}^{2}\nonumber\\
			&= \frac{11}{18} a_{1}^{2}+(-d_{k+2}+d_{k})a_{k}^{2},
		\end{align}
		which implies $\left\langle-L\eta_{0},\eta_{0}\right\rangle_{\rho}\geq0$ provided
		\begin{align}
			0<\ a_{k}^{2}\leq\ \frac{11}{18(d_{k+2}-d_{k})}a_{1}^{2}.
		\end{align}
		The proof of Corollary \ref{coro linear instability} is finished.
	\end{proof}

	\section{nonlinear instability}\label{sec nonlinear instability}
	
	In this section, we give the proof of Theorem \ref{the nonlinear instability} based on the linear instability  results obtained in Section 4. Recall that the nonlinear problem for $\eta = \omega + \sin2\theta$ can be written as
	\begin{equation}\label{nonlinear equ}
		\left\{\begin{split}
			&\partial_{t}\eta + L\eta =N(\eta),\\
			&\partial_{\theta}v = H\eta,\\
			&\eta(0,\theta)=\eta_{0}(\theta),\quad v(t,0) = 0,
		\end{split}\right.
	\end{equation}
	where
	\begin{align}
		&L\eta = \frac{1}{2}\sin2\theta\partial_{\theta}\eta - \cos2\theta\eta + \sin2\theta H\eta - 2\cos2\theta v,\\
		&N(\eta)=\partial_{\theta}v\eta - v\partial_{\theta}\eta.
	\end{align}
	
	\subsection{Existence and uniqueness of the nonlinear equation}\label{nonlinear existence}
	
	To begin with, we first establish the local well-posedness of the classical solution to  \ref{nonlinear equ} with odd initial data, which is
	
	\begin{lemma}\label{the local existence}
		Let  $\|\rho^{1/2}\partial_{\theta}\eta_{0}\|_{H^{m}}=\delta_{0}.$ Then there exists $T_{0}>0$ and a unique classical odd solution $\eta$ to \ref{nonlinear equ} such that
		\begin{align}
			\rho^{1/2}\partial_{\theta}\eta\in C([0,T_{0}];H^{m}(\mathbb{T}))\cap C^{2}([0,T_{0}];H^{m-2}(\mathbb{T})).
		\end{align}
		Moreover, setting $u:=\rho^{1/2}\partial_{\theta}\eta,$ the solution  satisfies
		\begin{align}
			\mathop{\sup}_{0\leq t\leq T_0}\|u(t)\|_{H^{m}},\mathop{\sup}_{0\leq t\leq T_0}\|\partial_{t}u(t)\|_{H^{m-1}},\mathop{\sup}_{0\leq t\leq T_0}\|\partial_{t}^{2}u(t)\|_{H^{m-2}}\leq C(T_0)\delta_{0},
		\end{align}
		where $T_0=\frac{1}{C_{{1}}}\ln(1+\frac{C_{{1}}}{2C_{{2}}\|u_{0}\|_{H^{m}}})$ with $C_{1},C_{2}>0$ being absolute constants (see \ref{C12} for details).
	\end{lemma}

	In addition, for general initial data, we have the following result.
	\begin{lemma}\label{the general existence}
		Assume that $\eta_{0}\in H^{m}(\mathbb{T})$. Then there exists $T > 0$ such that \ref{nonlinear equ} admits a unique classical solution satisfying
		\begin{align*}
			\eta\in C([0,T];H^{m}(\mathbb{T}))\cap C^{2}([0,T];H^{m-2}(\mathbb{T})).
		\end{align*}
	\end{lemma}
	
	Lemma \ref{the local existence} constitutes the first part of Theorem \ref{the nonlinear instability}. Its proof follows from Galerkin’s method, similar to the approach used in Theorem~\ref{the existence}. For convenience, we continue to denote the approximate solutions by $(\eta_{n},v_{n})$.Building on the linear analysis in Section \ref{exist}, the main task is to establish uniform estimates for the nonlinear terms. Since the derivative of the nonlinear term is
	\begin{align*}	\sqrt{\pi}\rho^{1/2}\partial_{\theta}N(\eta_{n})=\frac{1}{2\sin\theta}(\partial_{\theta}^{2}v_{n}\eta_{n}-v_{n}\partial_{\theta}^{2}\eta_{n}),
	\end{align*}
	which exhibits a stronger singularity, to derive uniform $H^m$ estimates for $\rho^{1/2} \partial_\theta \eta_n$, we rewrite $\rho^{1/2} \partial_\theta N(\eta_n)$ using the explicit form of the approximate solutions $\eta_n$.
	The proof of Lemma \ref{the general existence} is  derived by constructing approximate solutions through general basis $\{\sin k\theta, k\geq 1\}\cup\{\cos k\theta, k\geq 0\}$, which  is analogous to that of Lemma \ref{the local existence} and we omit the details here. We now present the proof of Lemma~\ref{the local existence}.
	
	\begin{proof}[Proof of Lemma \ref{the local existence}]
		The proof is based on Galerkin’s method and proceeds in several steps:
		
		\textbf{Step 1. Construction of the approximate solutions.}
		
		For a fixed positive integer $n$, we define
		\begin{align}
			&\eta_{n}=\sum_{k=1}^{n}\tilde{\eta}_{k}^{(o)}(t)\tilde{e}_{k}^{(o)},\label{etan expression}\\
			&\partial_{\theta}v_{n}=H\eta_{n}.\label{vn expression}
		\end{align}
		Our objective is to determine the coefficients $\tilde{\eta}_{k}^{(o)}(t)$ such that
		\begin{align}\label{etak ODE}
			\big\langle\partial_{t}\partial_{\theta}\eta_{n},\partial_{\theta}\tilde{e}_{k}^{(o)}\rho\big\rangle+\big\langle\partial_{\theta}L\eta_{n},\partial_{\theta}\tilde{e}_{k}^{(o)}\rho\big\rangle=\big\langle\partial_{\theta}N(\eta_{n}),\partial_{\theta}\tilde{e}_{k}^{(o)}\rho\big\rangle,
		\end{align}
		with the initial condition
		\begin{align}
			\big\langle\partial_{\theta}\eta_{0},\partial_{\theta}\tilde{e}_{k}^{(o)}\rho\big\rangle=\tilde{\eta}_{k}^{(o)}(0)
		\end{align}
		for  $k=1,2,\cdots,n$.                                                                                                                                                                                                                                                                                                                                                                                                                                                                                                                                                                                                                                       
		Following the notations introduced in Section \ref{sec linear instability}, we define
		\begin{align*}
			u_{n}=-\sqrt{\pi}\rho^{1/2}\partial_{\theta}\eta_{n}=\sum_{k=1}^{n}\tilde{\eta}_{k}^{(o)}(t)\sin(k+1)\theta.
		\end{align*}		
		To treat the nonlinear term, we introduce the notation
		\begin{align}\label{4161}
			N_{1}(u_{n}):=-\sqrt{\pi}\rho^{1/2}\partial_{\theta}N(\eta_{n})=\sqrt{\pi}\rho^{1/2}(v_{n}\partial_{\theta}^{2}\eta_{n}-\partial_{\theta}^{2}v_{n}\eta_{n}).
		\end{align}
		For the second term on the right-hand of \eqref{4161}, we  apply the fact that  $\partial_{\theta}v_{n}=H\eta_{n}$ to obtain
		\begin{align*}
			\sqrt{\pi}\rho^{1/2}\partial_{\theta}^{2}v_{n}=\frac{1}{2\sin\theta}\sum_{k=1}^{n}\tilde{\eta}_{k}^{(o)}(t)\left[\sin(k+2)\theta-\sin(k\theta)\right]=\sum_{k=1}^{n}\tilde{\eta}_{k}^{(o)}(t)\cos(k+1)\theta:=g_{n}.
		\end{align*}
		This implies that
		\begin{align}
			\|g_{n}\|_{H^{m}}=\|u_{n}\|_{H^{m}}
		\end{align}
		for any $m\geq 0$.
		
		Concerning the first term on the right-hand of \eqref{4161}, we represent $\sqrt{\pi}\rho^{1/2}v_{n}$ as an explicit  Fourier series.		Since $\eta_{n}$ and $v_{n}$ are odd functions, we use \ref{etan expression} and \ref{vn expression} to obtain
		
		\begin{align}\label{4162}
			v_{n}=\sum_{k=1}^{n+2}\frac{-\tilde{\eta}_{k-2}^{(o)}(t)+\tilde{\eta}_{k}^{(o)}(t)}{k^{2}}\sin k\theta,
		\end{align}
		where $\tilde{\eta}_{-1}^{(o)}(t), \tilde{\eta}_{0}^{(o)}(t),\tilde{\eta}_{n+1}^{(o)}(t)$ and $\tilde{\eta}_{n+2}^{(o)}(t)$ are understood to be zero.
		
		We now analyze the expression $\frac{\sin(k\theta)}{\sin\theta}.$ ~\\
		For $k=2l-1, l=1,2,\cdots$, we have
		\begin{align*}
			\frac{\sin k\theta}{\sin\theta}&=\frac{\sin\theta-\sin\theta+\sin(3\theta)-\cdots-\sin(2l-3)\theta+\sin(2l-1)\theta}{\sin\theta}\\
			&=1+2\sum_{j=1}^{l-1}\cos(2j\theta).
		\end{align*}
		On the other hand, if $k=2l, l=1,2,\cdots$, then
		\begin{align*}
			\frac{\sin k\theta}{\sin\theta}&=\frac{\sin(2\theta)-\sin(2\theta)+\sin(4\theta)-\cdots-\sin(2l-2)\theta+\sin(2l\theta)}{\sin\theta}\\
			&=2\sum_{j=1}^{l}\cos(2j-1)\theta.
		\end{align*}
		Without loss of generality, we consider the case $n=2l-1$ in \eqref{4162} to derive
		\begin{align*}
			\sqrt{\pi}\rho^{1/2}v_{n}=&\frac{1}{2}\sum_{k=1}^{l+1}\frac{-\tilde{\eta}_{2k-3}^{(o)}+\tilde{\eta}_{2k-1}^{(o)}}{(2k-1)^{2}} + \sum_{k=1}^{l}\cos(2k\theta)\left(\sum_{j\geq k}^{l}\frac{-\tilde{\eta}_{2j-1}^{(o)}+\tilde{\eta}_{2j+1}^{(o)}}{(2j+1)^{2}}\right)\\
			&+\sum_{k=1}^{l}\cos(2k-1)\theta\left(\sum_{j\geq k}^{l}\frac{-\tilde{\eta}_{2j-2}^{(o)}+\tilde{\eta}_{2j}^{(o)}}{(2j)^{2}}\right)\\
			:=&h_{n}.
		\end{align*}
		Thus, $N_{1}(u_{n})$ can be expressed as
		\begin{align*}
			N_{1}(u_{n})=-g_{n}\eta_{n}+h_{n}\partial_{\theta}^{2}\eta_{n}=-g_{n}\eta_{n}-2\sin\theta h_{n}\partial_{\theta}u_{n}-2\cos\theta h_{n}u_{n}.
		\end{align*}
		Define
		\begin{align*}
			L_{1}(u_{n}):=-\sqrt{\pi}\rho^{1/2}\partial_{\theta}L\eta_{n}.
		\end{align*}
		Then the ordinary differential equations \ref{etak ODE} are  equivalent to
		\begin{align}\label{un ODE}
			\left\langle\partial_{t}u_{n},\sin(k+1)\theta\right\rangle +\left\langle L_{1}(u_{n}),\sin(k+1)\theta\right\rangle=\left\langle N_{1}(u_{n}),\sin(k+1)\theta\right\rangle.
		\end{align}
		By the standard existence theory for ordinary differential equations, there exist a $T(k)>0$ which may depend on $k$ and a unique set of absolutely continuous functions $\tilde{\eta}_{k}^{(o)}(t) (k=1,2,\cdots,n)$ satisfying \ref{etan expression} and \ref{vn expression} for $0\leq t\leq T(k)$.

		\textbf{Step 2. Energy estimates.}
		
		The estimates for the linear operator $L$ have been established in the Section \ref{exist}, and  we now focus on the estimates for the nonlinear terms.
		
		{ \it \underline{$L^{2}$ estimate of $u_{n}.$}}
		
		Multiplying \ref{un ODE} by $\tilde{\eta}_{k}^{(o)}(t)$ and summing over $k=1,2,\cdots,n$, we obtain
		\begin{align*}
			\frac{1}{2}\frac{d}{dt}\|u_{n}\|_{L^{2}}^{2}=-\left\langle L_{1}(u_{n}),u_{n}\right\rangle+\left\langle N_{1}(u_{n}),u_{n}\right\rangle.
		\end{align*}
		Noticing that $\|g_{n}\|_{H^{m}}=\|u_{n}\|_{H^{m}}$ for  any $m\geq 0$, and applying Poincar\'{e}'s inequality and H\"{o}lder's inequality, we obtain
		\begin{align*}
			\left\langle N_{1}(u_{n}),u_{n}\right\rangle&\lesssim \|\eta_{n}\|_{L^{\infty}}\|g_{n}\|_{L^{2}}\|u_{n}\|_{L^{2}}+\left(\|\partial_{\theta}(\sin\theta h_{n})\|_{L^{\infty}}+\|h_{n}\|_{L^{\infty}}\right)\|u_{n}\|_{L^{2}}^2\\
			&\lesssim \|\partial_{\theta}\eta_{n}\|_{L^{2}}\|u_{n}\|_{L^{2}}^2+\left(\|\partial_{\theta}(\sin\theta h_{n})\|_{L^{\infty}}+\|h_{n}\|_{L^{\infty}}\right)\|u_{n}\|_{L^{2}}^2.
		\end{align*}
		To estimate $\|h_{n}\|_{L^{\infty}}=\big\|\frac{v_{n}}{2\sin\theta}\big\|_{L^{\infty}}$, we note that  $v_{n}$ is odd and periodic and hence $v_{n}(\pi)=v_{n}(0)$.  Since $\sin\theta\geq\frac{2}{\pi}\min\{\theta,\pi-\theta\}$ in $[0,\pi]$, we have $\|h_{n}\|_{L^{\infty}}=\big\|\frac{v_{n}}{\sin\theta}\big\|_{L^{\infty}}\lesssim\|\partial_{\theta}v_{n}\|_{L^{\infty}}$ by Lagrange's mean value theorem. Then applying \ref{vn expression}, Poincar\'{e}'s inequality and Lemma \ref{le hilbert estimate}, we obtain
		\begin{align}
			\|h_{n}\|_{L^{\infty}}\lesssim\|H(\partial_{\theta}\eta_{n})\|_{L^{2}}\lesssim\|\partial_{\theta}\eta_{n}\|_{L^{2}}\lesssim\|u_{n}\|_{L^{2}}
		\end{align}
		and
		\begin{align}
			\|\partial_{\theta}(\sin\theta h_{n})\|_{L^{\infty}}=\bigg\|\partial_{\theta}v_{n}-\cos\theta\frac{v_{n}}{\sin\theta}\bigg\|_{L^{\infty}}\lesssim\|\partial_{\theta}v_{n}\|_{L^{\infty}}\lesssim\|u_{n}\|_{L^{2}}.
		\end{align}
		Thus, we obtain
		\begin{align*}
			\left\langle N_{1}(u_{n}),u_{n}\right\rangle\lesssim\|u_{n}\|_{L^{2}}^{3}.
		\end{align*}
		It follows that
		\begin{align}\label{duL}
			\frac{d}{dt}\|u_{n}(t)\|_{L^{2}}\leq C\|u_{n}(t)\|_{L^{2}}+C\|u_{n}(t)\|_{L^{2}}^{2},
		\end{align}
		where $C>0$ is a  constant.
		
		{\it \underline{$H^{1}$ estimate of $u_{n}.$}}

		Multiplying equation \ref{un ODE} by $(k+1)^{2}\tilde{\eta}_{k}^{(o)}(t)$, summing over $k=1,2,\cdots,n$, and applying the integration by parts, we get
		\begin{align*}
			\frac{1}{2}\frac{d}{dt}\|\partial_{\theta}u_{n}\|^2_{L^{2}}=-\left\langle \partial_{\theta}L_{1}(u_{n}),\partial_{\theta}u_{n}\right\rangle+\left\langle \partial_{\theta}N_{1}(u_{n}),\partial_{\theta}u_{n}\right\rangle,
		\end{align*}
		where
		\begin{align*}
			\partial_{\theta}N_{1}(u_{n})=&2\sin\theta h_{n}\partial_{\theta}^{2}u_{n}+\left(4\cos\theta h_{n}+2\sin\theta\partial_{\theta}h_{n}\right)\partial_{\theta}u_{n}\\
			&+2\left(-\sin\theta h_{n}+\cos\theta\partial_{\theta}h_{n}\right)u_{n}-\partial_{\theta}g_{n}\eta_{n}-g_{n}\partial_{\theta}\eta_{n}.
		\end{align*}
		Applying integration by parts, we obtain
		\begin{align*}
			\left\langle\sin\theta h_{n}\partial_{\theta}^{2}u_{n},\partial_{\theta}u_{n}\right\rangle\leq \|\partial_{\theta}(\sin\theta h_{n})\|_{L^{\infty}}\|\partial_{\theta}u_{n}\|_{L^{2}}^{2}\lesssim\|u_{n}\|_{H^{1}}^{3}.
		\end{align*}
		For the term $\left(-\sin\theta h_{n}+\cos\theta\partial_{\theta}h_{n}\right)u_{n}$,
		\begin{align*}
			\left\langle\left(-\sin\theta h_{n}+\cos\theta\partial_{\theta}h_{n}\right)u_{n},\partial_{\theta}u_{n}\right\rangle&=\left\langle -\sin\theta h_{n}u_{n},\partial_{\theta}u_{n}\right\rangle+\big\langle \frac{u_{n}}{\sin\theta}\partial_{\theta}v_{n}+\frac{v_{n}\cos\theta}{\sin\theta}\frac{u_{n}}{\sin\theta},\partial_{\theta}u_{n}\big\rangle\\
			&\lesssim\|h_{n}\|_{L^{\infty}}\|u_{n}\|_{L^{2}}\|\partial_{\theta}u_{n}\|_{L^{2}} + \big\|\frac{u_{n}}{\sin\theta}\big\|_{L^{\infty}}\|\partial_{\theta}v_{n}\|_{L^{2}}\|\partial_{\theta}u_{n}\|_{L^{2}}\\
			&\lesssim\|u_{n}\|_{H^{1}}^{3}.
		\end{align*}
		Then, it follows that
		\begin{align*}
			\left\langle \partial_{\theta}N_{1}(u_{n}),\partial_{\theta}u_{n}\right\rangle\lesssim\|u_{n}\|_{H^{1}}^{3}.
		\end{align*}
		Consequently, we obtain
		\begin{align}\label{0502}
			\frac{d}{dt}\|u_{n}(t)\|_{H^{1}}\leq C\|u_{n}(t)\|_{H^{1}}+C\|u_{n}(t)\|_{H^{1}}^{2},
		\end{align}
		where $	C>0$ is a constant.

		{\it \underline{$H^{m}(m\geq 2)$ estimate of $u_{n}.$}}
		
		Multiplying \ref{un ODE} by $(k+1)^{2m}\tilde{\eta}_{k}^{(o)}(t)$, summing up over $k=1,2,\cdots,n$ and using the integration by parts, we obtain
		\begin{align*}
			\frac{1}{2}\frac{d}{dt}\|\partial_{\theta}^{m}u_{n}\|^2_{L^{2}}=-\left\langle\partial_{\theta}^{m}L_{1}(u_{n}),\partial_{\theta}^{m}u_{n}\right\rangle+\left\langle\partial_{\theta}^{m}N_{1}(u_{n}),\partial_{\theta}^{m}u_{n}\right\rangle.
		\end{align*}
		To get the estimate of $\left\langle\partial_{\theta}^{m}N_{1}(u_{n}),\partial_{\theta}^{m}u_{n}\right\rangle$, we first provide an estimate for $\|\partial_{\theta}^{m}h_{n}\|_{L^{2}}$. Direct calculations give that
		\begin{align*}
			\|\partial_{\theta}^{m}h_{n}\|_{L^{2}}^{2}&=\sum_{k=1}^{l}(2k)^{2m}\left(\sum_{j\geq k}^{l}\frac{-\tilde{\eta}_{2j-1}^{(o)}+\tilde{\eta}_{2j+1}^{(o)}}{(2j+1)^{2}}\right)^{2} + \sum_{k=1}^{l}(2k-1)^{2m}\left(\sum_{j\geq k}^{l}\frac{-\tilde{\eta}_{2j-2}^{(o)}+\tilde{\eta}_{2j}^{(o)}}{(2j)^{2}}\right)^{2}\\
			&\leq\sum_{k=1}^{l}(2k)^{2m}\sum_{j\geq k}^{l}\left((\tilde{\eta}_{2j-1}^{(o)})^{2}\cdot\frac{1}{(2j+1)^{4}}\right) + \sum_{k=1}^{l}(2k-1)^{2m}\sum_{j\geq k}^{l}\left((\tilde{\eta}_{2j-2}^{(o)})^{2}\cdot\frac{1}{(2j)^{4}}\right)\\
			&\leq \sum_{j\geq 1}^{l}\sum_{k\leq j}\left((\tilde{\eta}_{2j-1}^{(o)})^{2}\cdot\frac{(2k)^{2m}}{(2j+1)^{4}}\right) + \sum_{j\geq 1}^{l}\sum_{k\leq j}\left((\tilde{\eta}_{{2j-2}}^{(o)})^{2}\cdot\frac{(2k-1)^{2m}}{(2j)^{4}}\right)\\
			&\leq \sum_{j\geq 1}^{l}(\tilde{\eta}_{2j-1}^{(o)})^{2}\cdot j\cdot\frac{(2j)^{2m}}{(2j+1)^{4}} +\sum_{j\geq 1}^{l}(\tilde{\eta}_{2j-2}^{(o)})^{2}\cdot j\cdot\frac{(2j-1)^{2m}}{(2j)^{4}}\\
			& \leq \sum_{k=1}^{n}(k+1)^{2m-4+1}(\tilde{\eta}_{k}^{(o)})^{2}=\sum_{k=1}^{n}(k+1)^{2m-3}(\tilde{\eta}_{k}^{(o)})^{2}\\
			&\leq \sum_{k=1}^{n}(k+1)^{2m-2}(\tilde{\eta}_{k}^{(o)})^{2}\\
			& =\|\partial_{\theta}^{m-1}u_{n}\|^2_{L^{2}},
		\end{align*}
		which means that
		\begin{align}\label{hn estimtes}
			\|\partial_{\theta}^{m}h_{n}\|_{L^{2}}\leq \|\partial_{\theta}^{m-1}u_{n}\|_{L^{2}}
		\end{align}
		for $m\geq1$. For a canonical term in $\partial_{\theta}^{m}N_{1}(u_{n})$, it is of the form $\partial_{\theta}^{m_{1}}g_{n}\partial_{\theta}^{m_{2}}\eta_{n}$ or  $\partial_{\theta}^{m_{1}+1}u_{n}\partial_{\theta}^{m_{2}}h_{n}\partial_{\theta}^{m_{3}}\sin\theta$ or  $\partial_{\theta}^{m_{1}}u_{n}\partial_{\theta}^{m_{2}}h_{n}\partial_{\theta}^{m_{3}}\cos\theta$. For the term $\partial_{\theta}^{m+1}u_{n}h_{n}\sin\theta$, we apply integration by parts to obtain
		\begin{align*}
			\left\langle \sin\theta h_{n}\partial_{\theta}^{m+1}u_{n},\partial_{\theta}^{m}u_{n}\right\rangle&=-\frac{1}{2}\left\langle\partial_{\theta}(h_{n}\sin\theta),(\partial_{\theta}^{m}u_{n})^{2}\right\rangle\\
			&\leq\|\partial_{\theta}(h_{n}\sin\theta)\|_{L^{\infty}}\|\partial_{\theta}^{m}u_{n}\|_{L^{2}}^{2}\lesssim\|u_{n}\|_{L^{2}}\|\partial_{\theta}^{m}u_{n}\|_{L^{2}}^{2}.
		\end{align*}
		For the terms involving $\partial_{\theta}^{m}g_{n}\eta_{n}$, $\partial_{\theta}^{m}u_{n}h_{n}\cos\theta$ and $\partial_{\theta}^{m}\partial_{\theta}h_{n}\sin\theta$, applying Poincar\'{e}'s inequality yields
		\begin{align*}
			\left\langle\partial_{\theta}^{m}g_{n}\eta_{n},\partial_{\theta}^{m}u_{n}\right\rangle\leq\|\eta_{n}\|_{L^{\infty}}\|\partial_{\theta}^{m}u_{n}\|_{L^{2}}^{2}\lesssim\|u_{n}\|_{L^{2}}\|\partial_{\theta}^{m}u_{n}\|_{L^{2}}^{2}
		\end{align*}
		and
		\begin{align*}
			\left\langle(m+1)\cos\theta h_{n}\partial_{\theta}^{m}u_{n}+m\sin\theta\partial_{\theta}^{m}u_{n}\partial_{\theta}h_{n},\partial_{\theta}^{m}u_{n}\right\rangle&\lesssim\left(\|h_{n}\|_{L^{\infty}}+\|\partial_{\theta}h_{n}\sin\theta\|_{L^{\infty}}\right)\|\partial_{\theta}^{m}u_{n}\|_{L^{2}}^{2}\\
			&\lesssim\|u_{n}\|_{L^{2}}\|\partial_{\theta}^{m}u_{n}\|_{L^{2}}^{2}.
		\end{align*}
		We now focus on controlling the $L^{2}-$norms of $\partial_{\theta}^{m_{1}+1}u_{n}\partial_{\theta}^{m_{2}}h_{n}\partial_{\theta}^{m_{3}}\sin\theta$ for indices $0\leq m_{1}\leq m-2$, $\partial_{\theta}^{m_{1}}g_{n}\partial_{\theta}^{m_{2}}\eta_{n}$ for indices $0\leq m_{1}\leq m-1$ and $\partial_{\theta}^{m_{1}}u_{n}\partial_{\theta}^{m_{2}}h_{n}\partial_{\theta}^{m_{3}}\sin\theta$ for indices $0\leq m_{1}\leq m-1$. For example, applying Poincar\'{e}'s inequality and the estimate \ref{hn estimtes} we obtain
		\begin{align*}
			\|\partial_{\theta}^{m_{1}}u_{n}\partial_{\theta}^{m_{2}}h_{n}\partial_{\theta}^{m_{3}}\sin\theta\|_{L^{2}}&\leq\|\partial_{\theta}^{m_{1}}u_{n}\|_{L^{\infty}}\|\partial_{\theta}^{m_{2}}h_{n}\|_{L^{2}}\\
			&\lesssim\|\partial_{\theta}^{m_{1}+1}u_{n}\|_{L^{2}}\|\partial_{\theta}^{m_{2}-1}u_{n}\|_{L^{2}}\lesssim\|u_{n}\|_{H^{m_{1}+1}}\|u_{n}\|_{H^{m_{2}-1}}.
		\end{align*}
		It concludes that
		\begin{align*}
			\left\langle\partial_{\theta}^{m}N_{1}(u_{n}),\partial_{\theta}^{m}u_{n}\right\rangle=&-\big\langle m!\sum_{\substack{m_{1}+m_{2}=m\\m_{1},m_{2}\geq0}}\frac{\partial_{\theta}^{m_{1}}g_{n}}{m_{1}!}\frac{\partial_{\theta}^{m_{2}}\eta_{n}}{m_{2}!}\big\rangle\\
			&+2\big\langle m!\sum_{\substack{m_{1}+m_{2}+m_{3}=m\\m_{1},m_{2},m_{3}\geq0}}\frac{\partial_{\theta}^{m_{1}+1}u_{n}}{m_{1}!}\frac{\partial_{\theta}^{m_{2}}h_{n}}{m_{2}!}\frac{\partial_{\theta}^{m_{3}}\sin\theta}{m_{3}!},\partial_{\theta}^{m}u_{n}\big\rangle \\
			&+2\big\langle m!\sum_{\substack{m_{1}+m_{2}+m_{3}=m\\m_{1},m_{2},m_{3}\geq0}}\frac{\partial_{\theta}^{m_{1}}u_{n}}{m_{1}!}\frac{\partial_{\theta}^{m_{2}}h_{n}}{m_{2}!}\frac{\partial_{\theta}^{m_{3}}\cos\theta}{m_{3}!},\partial_{\theta}^{m}u_{n}\big\rangle\\
			\leq&C\|u_{n}\|_{H^{m}}^{3}.
		\end{align*}
		In view of \ref{duL} and \ref{0502}, we obtain
		\begin{align}\label{C12}
			\frac{d}{dt}\|u_{n}(t)\|_{H^{m}}\leq C_1\|u_{n}(t)\|_{H^{m}}+C_2\|u_{n}(t)\|_{H^{m}}^{2},
		\end{align}
		where $C_1, C_2$ are positive constants independent of $n$. A direct computation yields 
		\begin{align}\label{050201}
			\|u_{n}(t)\|_{H^{m}}\leq\frac{C_{{1}}}{C_{{2}}(1-e^{C_{{1}}t})\|u_{0}\|_{H^{m}}+C_{{1}}}\|u_{0}\|_{H^{m}}e^{C_{{1}}t}
		\end{align}
		for $0\leq t<T^{*}$ with $T^*=\frac{1}{C_{{1}}}\ln(1+\frac{C_{{1}}}{C_{{2}}\|u_{0}\|_{H^{m}}})$ for any fixed integer $m\ge1$.
		
		We choose
		$$
		T_0=\frac{1}{C_{{1}}}\ln(1+\frac{C_{{1}}}{2C_{{2}}\|u_{0}\|_{H^{m}}}).
		$$
		It then follows that
		\begin{align}\label{u Hm}
			\mathop{\sup}_{0\leq t\leq T_{{0}}}\|u_{n}(t)\|_{H^{m}}\leq\ 2e^{C_{{1}}T_{{0}}}\|u_{0}\|_{H^{m}}.
		\end{align}	
		
		{\it \underline{High estimates of $\partial_{t}u_{n}$ and $\partial_{t}^{2}u_{n}.$}}
		
		To estimate $\partial_{t}u_{n}$, we multiply equation \ref{un ODE} by $(k+1)^{2m-2}\frac{d}{dt}\tilde{\eta}_{k}^{(o)}(t)$ and sum up over $k=1,2,\cdots,n$. After applying integration by parts, we obtain
		\begin{align*}
			\|\partial_{t}\partial_{\theta}^{m-1}u_{n}\|_{L^{2}}=-\left\langle\partial_{\theta}^{m-1}L_{1}(u_{n}),\partial_{t}\partial_{\theta}^{m-1}u_{n}\right\rangle+\left\langle\partial_{\theta}^{m-1}N_{1}(u_{n}),\partial_{t}\partial_{\theta}^{m-1}u_{n}\right\rangle.
		\end{align*}
		For the term $\partial_{\theta}^{m-1}N_{1}(u_{n})$, we have
		\begin{align*}
			\|\partial_{\theta}^{m-1}N_{1}(u_{n})\|_{L^{2}}\lesssim\|u_{n}\|_{H^{m}}^{2}.
		\end{align*}
		Combining with \ref{050201}, we have
		\begin{align}\label{ptun hm}
			\mathop{\sup}_{0\leq t\le T_{0}}\| \partial_{t}u_{n}(t)\|_{H^{m-1}} \leq\ Ce^{2 C_{{1}}T_{{0}}}\|u_{0}\|_{H^{m}}^{2},
		\end{align}
		where $C$ is a positive constant independent of $n$.
		Similarly, to get the estimate of $\partial_{t}^{2}u_{n}$, we apply $\partial_{t}$ to equation \ref{un ODE}, multiply $(k+1)^{2m-4}\frac{d^{2}}{dt^{2}}\tilde{\eta}_{k}^{(o)}(t)$, sum up over $k=1,2,\cdots,n$ and integrate by parts to obtain
		\begin{align*}
			\|\partial_{t}^{2}\partial_{\theta}^{m-2}u_{n}\|_{L^{2}}^{2}=-\left\langle\partial_{t}\partial_{\theta}^{m-2}L_{1}(u_{n}),\partial_{t}^{2}\partial_{\theta}^{m-2}u_{n}\right\rangle+\left\langle\partial_{t}\partial_{\theta}^{m-2}N_{1}(u_{n}),\partial_{t}^{2}\partial_{\theta}^{m-2}u_{n}\right\rangle.
		\end{align*}
		By combining the uniform estimates \ref{u Hm}, we deduce that
		\begin{align}\label{pttun hm}
			\mathop{\sup}_{0\leq t\le T_{0}}\| \partial_{t}^{2}u_{n}(t)\|_{H^{m-2}} \leq\ Ce^{2 C_{{1}}T_{{0}}}\|u_{0}\|_{H^{m}}^{2},
		\end{align}
		where $C$ is a positive constant independent of $n$.

		\textbf{Step 3. Existence and uniqueness}
		
		By combing the estimates \ref{u Hm}, \ref{ptun hm} and \ref{pttun hm}, the existence interval $[0,T(k)]$ established in Step 1 of Subsection \ref{nonlinear existence} can be extended to  $[0,T_0]$. Furthermore, there exists a subsequence $\{u_{n_{k}}\}_{n_{k}=1}^{\infty}\subseteq \{u_{n}\}_{n=1}^{\infty}$
		\begin{align*}
			&u_{n_{k}}\rightharpoonup u\ weakly\ in\ L^{2}([0,T_0];H^{m}(\mathbb{T})),\\
			&\partial_{t}u_{n_{k}}\rightharpoonup\partial_{t}u\ weakly\ in\ L^{2}([0,T_0];H^{m-1}(\mathbb{T})),\\
			&\partial_{t}^{2}u_{n_{k}}\rightharpoonup\partial_{t}^{2}u\ weakly\ in\ L^{2}([0,T_0];H^{m-2}(\mathbb{T})).
		\end{align*}
		The left argument is similar to that of Theorem 1.1 (see Step 3 in the proof of Theorem 1.1) and we omit it here.
		
		Consequently, combining \eqref{u Hm}, \eqref{ptun hm} and \eqref{pttun hm}, we obtain
		\begin{align}
			\mathop{\sup}_{0\leq t\leq T_0}\|u(t)\|_{H^{m}},\mathop{\sup}_{0\leq t\leq T_0}\|\partial_{t}u(t)\|_{H^{m-1}},\mathop{\sup}_{0\leq t\leq T_0}\|\partial_{t}^{2}u(t)\|_{H^{m-2}}\leq C(T_0)\delta_{0}.
		\end{align}
		This completes the proof of Lemma \ref{the local existence}.
	\end{proof}

	\subsection{Nonlinear instability}\label{subsection nonlinear instability}
	
	We now address the second part of Theorem \ref{the nonlinear instability}. To this end, we first construct a solution to the linearized problem and a family solutions to the nonlinear problem that possess certain special properties. We then apply a contradiction argument to establish the existence of a nonlinear solution satisfying the conditions of Theorem \ref{the nonlinear instability} and exhibiting instability.
	
	Let $\delta>0, K>0$ and $F$ satisfying \ref{def F} be arbitrary but given.
	\begin{proof}[Proof of Theorem \ref{the nonlinear instability}]
		
		\quad
		
		\textbf{Step 1. Construction of a solution to the linearized problem.}
		
		By Theorems \ref{the existence}-\ref{the linear instability}, a classical solution $(\eta,v)$ to the linearized equation \ref{etav} can be constructed. Assume that the initial data satisfy the conditions of Theorem \ref{the existence}, and, in addition, fulfill
		\begin{align}\label{L condition}
			0\leq \left\langle-L\eta_{0},\eta_{0}\right\rangle_{\rho}\leq\sqrt{\lambda_{1}}\left\langle\eta_{0},\eta_{0}\right\rangle_{\rho}.
		\end{align}
		Then, it concludes that
		\begin{align}
			\|\eta\|_{\mathcal{H}_{DW}}>\frac{1}{2}\|\eta_{0}\|_{\mathcal{H}_{DW}}e^{\sqrt{\lambda_{1}}t}>0
		\end{align}
		for all $t>0$.
		To simplify the notation, we define $u:=\rho^{1/2}\partial_{\theta}\eta$, so that $\|u\|_{L^{2}}=\|\rho^{1/2}\partial_{\theta}\eta\|_{L^{2}}=\|\eta\|_{\mathcal{H}_{DW}}$. We further define
		\begin{align}
			\breve{\eta}:=\frac{\delta\eta}{\|u_{0}\|_{H^{m}}},
		\end{align}
		and
		\begin{align}\label{def breve u}
			\breve{u}:=\rho^{1/2}\partial_{\theta}\breve{\eta}=\frac{\delta u}{\|u_{0}\|_{H^{m}}}.
		\end{align}
		It follows that $(\breve{\eta},\breve{v})$ remains a classical solution to the linearized equation \ref{etav}, with the same properties as $(\eta,v)$. Specifically, we have
		\begin{align}\label{breeta prop}
			\|\breve{\eta}(t)\|_{\mathcal{H}_{DW}}>\frac{1}{2}\|\breve{\eta}_{0}\|_{\mathcal{H}_{DW}}e^{\sqrt{\lambda_{1}}t}>0.
		\end{align}
		Moreover, we have
		\begin{align}
			\|\breve{u}(0)\|_{H^{m}}=\delta.
		\end{align}
		Let
		\begin{align}
			t_{K}:=\frac{1}{\sqrt{\lambda_{1}}}\ln\frac{4K\delta}{\|\breve{u}_{0}\|_{L^{2}}}.
		\end{align}
		Then, it follows from \ref{instability25} that
		\begin{align}\label{breve eta asumption}
			\|\breve{u}(t_{K})\|_{L^{2}}=\|\breve{\eta}(t_{K})\|_{\mathcal{H}_{DW}}\geq J^{1/2}_{1}(t_{K})\geq \frac{1}{2}e^{t_{K}\sqrt{\lambda_{1}}}\|\breve{\eta}_{0}\|_{H_{DW}}= 2K\delta.
		\end{align}
		
		\textbf{Step 2. Construction of a solution to the corresponding nonlinear problem.}
		
		Based on the initial data $\breve{u}(0)$ of the solution $\breve{u}$ defined in \ref{def breve u}, we proceed to construct a family solutions to the perturbed nonlinear problem.\\
		Let
		\begin{align}
			\bar{\eta}^{\varepsilon}_{0}:=\varepsilon\breve{\eta}_{0}
		\end{align}
		and
		\begin{align*}
			\bar{u}^{\varepsilon}_{0}:=\rho^{1/2}\partial_{\theta}\bar{\eta}^{\varepsilon}_{0}=\varepsilon\breve{u}_{0}
		\end{align*}
		for $\varepsilon\in(0,1)$.
		Then it follows that
		\begin{align}
			\bar{u}_{0}^{\varepsilon}=\rho^{1/2}\partial_{\theta}\bar{\eta}_{0}^{\varepsilon}\in H^{m}
		\end{align}
		for $m>3,$ and
		\begin{align}\label{ueps initial data}
			\|\bar{u}^{\varepsilon}_{0}\|_{H^{m}}=\delta\varepsilon<\delta.
		\end{align}
		By Lemma \ref{the local existence}, there exists a constant $\varepsilon_1>0$ such that for all $\varepsilon \in (0, \varepsilon_1)$, the nonlinear equation \ref{nonlinear equ} admits a classical solution $(\bar{\eta}^{\varepsilon}, \bar{v}^{\varepsilon})$ on $(0, T_{\varepsilon})$, where
		\begin{align*}
			T_{\varepsilon}=\frac{1}{C_{{1}}}\ln(1+\frac{C_{{1}}}{2C_{{2}}\delta\varepsilon}) > t_K.
		\end{align*}
		Moreover, we have
		\begin{align}\label{tk u estimate}
			\mathop{\sup}_{0\leq t\leq t_{K}}\|\bar{u}^{\varepsilon}(t)\|_{H^{m}},\mathop{\sup}_{0\leq t\leq t_{K}}\|\partial_{t}\bar{u}^{\varepsilon}(t)\|_{H^{m-1}},\mathop{\sup}_{0\leq t\leq t_{K}}\|\partial_{t}^{2}\bar{u}^{\varepsilon}(t)\|_{H^{m-2}}\leq C(t_{K})\delta\varepsilon,
		\end{align}
		where $\bar{u}^{\varepsilon}:=\rho^{1/2}\partial_{\theta}\bar{\eta}^{\varepsilon}.$
		
		Next, we proceed to the proof of Theorem \ref{the nonlinear instability}. Suppose, for the sake of contradiction, for any $\varepsilon\in(0,\varepsilon_{1})$, the classical solution $\bar{\eta}^{\varepsilon}$, emanating from the initial data $\bar{\eta}^{\varepsilon}_{0},$ satisfies
		\begin{align*}
			\|\bar{\eta}^{\varepsilon}\|_{\mathcal{H}_{DW}}=\|\bar{u}^{\varepsilon}\|_{L^{2}}\leq F(\|\bar{u}^{\varepsilon}_{0}\|_{H^{m}})
		\end{align*}
		for any $t\in(0,t_{K}]\subset(0,T_{\varepsilon})$, where $F$ is the function defined in \ref{def F}. Combining this with \ref{ueps initial data}, we obtain
		\begin{align}
			\mathop{\sup}_{0\leq t\leq t_{K}}\|\bar{u}^{\varepsilon}(t)\|_{L^{2}}\leq K\|\bar{u}^{\varepsilon}_{0}\|_{H^{m}}\leq K\delta\varepsilon.
		\end{align}
		We denote
		\begin{align}
			(\tilde{\eta}^{\varepsilon},\tilde{v}^{\varepsilon}):=(\bar{\eta}^{\varepsilon},\bar{v}^{\varepsilon})/\varepsilon,
		\end{align}
		and
		\begin{align}
			\tilde{u}^{\varepsilon}:=\rho^{1/2}\partial_{\theta}\tilde{\eta}^{\varepsilon}.
		\end{align}
		Then $(\tilde{\eta}^{\varepsilon},\tilde{v}^{\varepsilon})$ satisfies
		\begin{equation}\label{bareta equ}
			\left\{\begin{split}
				&\partial_{t}(\varepsilon\tilde{\eta}^{\varepsilon}) +L(\varepsilon\tilde{\eta}^{\varepsilon})=N(\varepsilon\tilde{\eta}^{\varepsilon}),\\
				&\partial_{\theta}(\varepsilon\tilde{v}^{\varepsilon})=H(\varepsilon\tilde{\eta}^{\varepsilon}),
			\end{split}\right.
		\end{equation}
		with the initial data
		\begin{align}\label{initial data}
			\tilde{\eta}^{\varepsilon}(0)=\frac{1}{\varepsilon}\bar{\eta}^{\varepsilon}(0)=\breve{\eta}_{0}.
		\end{align}
		Here \ref{bareta equ} and \ref{initial data} can be rewritten as
		\begin{equation}\label{bareta equ1}
			\left\{\begin{split}
				&\partial_{t}\tilde{\eta}^{\varepsilon} +L\tilde{\eta}^{\varepsilon}=\varepsilon N(\tilde{\eta}^{\varepsilon}),\\
				&\partial_{\theta}\tilde{v}^{\varepsilon}=H\tilde{\eta}^{\varepsilon}, \quad \tilde{\eta}^{\varepsilon}(0)=\breve{\eta}_{0},
			\end{split}\right.
		\end{equation}
		where $L$ is the linearized operator given by
		\begin{align*}
			L\tilde{\eta}^{\varepsilon}=\frac{1}{2}\sin2\theta\partial_{\theta}\tilde{\eta}^{\varepsilon}-\cos2\theta\tilde{\eta}^{\varepsilon}+\sin2\theta H\tilde{\eta}^{\varepsilon}-2\cos2\theta\tilde{v}^{\varepsilon},
		\end{align*}
		and $N(\tilde{\eta}^{\varepsilon})$ denotes the nonlinear term
		\begin{align*}
			N(\tilde{\eta}^{\varepsilon})=\partial_{\theta}\tilde{v}^{\varepsilon}\tilde{\eta}^{\varepsilon}-\tilde{v}^{\varepsilon}\partial_{\theta}\tilde{\eta}^{\varepsilon}.
		\end{align*}
		Moreover, the following estimates hold due to \ref{tk u estimate}:
		\begin{align}
			&\mathop{\sup}_{0\leq t\leq t_{K}}\|\tilde{u}^{\varepsilon}(t)\|_{H^{m}}\leq\ C(t_{K})\delta,\\
			&\mathop{\sup}_{0\leq t\leq t_{K}}\|\partial_{t}\tilde{u}^{\varepsilon}(t)\|_{H^{m-1}}, \mathop{\sup}_{0\leq t\leq t_{K}}\|\partial_{t}^{2}\tilde{u}^{\varepsilon}(t)\|_{H^{m-2}}\leq\ C(t_{K})\delta,
		\end{align}
		which are independent of $\varepsilon.$
		Thus, we immediately infer that there exists a subsequence (not relabeled) of $\{\bar{\eta}^{\varepsilon}\}$ such that
		\begin{align*}
			&\tilde{u}^{\varepsilon}\rightharpoonup \tilde{u}\ weakly\ in\ L^{2}([0,t_{K}];H^{m}(\mathbb{T})),\\
			&\partial_{t}\tilde{u}^{\varepsilon}\rightharpoonup\partial_{t}\tilde{u}\ weakly\ in\ L^{2}([0,t_{K}];H^{m-1}(\mathbb{T})),\\
			&\partial_{t}^{2}\tilde{u}^{\varepsilon}\rightharpoonup\partial_{t}^{2}\tilde{u}\ weakly\ in\ L^{2}([0,t_{K}];H^{m-2}(\mathbb{T})),
		\end{align*}
		and
		\begin{align}\label{tilde u assum}
			\mathop{\sup}_{0\leq t\leq t_{K}}\|\tilde{u}(t)\|_{L^{2}}\leq K\delta,\ \tilde{u}\in C([0,t_{K}],H^{m}(\mathbb{T}))\cap C^{2}([0,t_{K}],H^{m-2}(\mathbb{T})).
		\end{align}
		In fact, since $\tilde{\eta}^{\varepsilon}$ is odd and satisfies $\partial_{\theta}\tilde{\eta}^{\varepsilon}=\rho^{-1/2}\tilde{u}^{\varepsilon}$, we also have
		\begin{align*}
			&\tilde{\eta}^{\varepsilon}\rightharpoonup \tilde{\eta}\ weakly\ in\ L^{2}([0,t_{K}];H^{m+1}(\mathbb{T})),\\
			&\partial_{t}\tilde{\eta}^{\varepsilon}\rightharpoonup\partial_{t}\tilde{\eta}\ weakly\ in\ L^{2}([0,t_{K}];H^{m}(\mathbb{T})),\\
			&\partial_{t}^{2}\tilde{\eta}^{\varepsilon}\rightharpoonup\partial_{t}^{2}\tilde{\eta}\ weakly\ in\ L^{2}([0,t_{K}];H^{m-1}(\mathbb{T})).
		\end{align*}
		Passing the limit $\varepsilon\to0$ in the equations \ref{bareta equ1}, we arrive at the linearized model
		\begin{equation}\label{tilde eta equ}
			\left\{\begin{split}
				&\partial_{t}(\tilde{\eta}) +L(\tilde{\eta})=0,\\
				&\partial_{\theta}\tilde{v}=H\tilde{\eta}.
			\end{split}\right.
		\end{equation}
		Thus, $(\tilde{\eta},\tilde{v})$ is a classical solution to the linearized problem \ref{etav} on $[0,t_{K}],$ with the same initial data as $\breve{\eta}$.  Therefore, by Theorem \ref{the existence}, we conclude that
		\begin{align}
			\breve{\eta} (t,\theta) = \tilde{\eta} (t,\theta), \quad on\ [0,t_{K}]\times\mathbb{T}.
		\end{align}
		Combining the assumption \ref{breve eta asumption} with \ref{tilde u assum} , we deduce that
		\begin{align*}
			2K\delta\leq\|\breve{\eta}(t_{K})\|_{\mathcal{H}_{DW}}=\|\tilde{\eta}(t_{K})\|_{\mathcal{H}_{DW}}\leq K\delta,
		\end{align*}
		which yields a contradiction. 
		
		Thus we complete the proof of Theorem \ref{the nonlinear instability}.
	\end{proof}

	\section{Nonlinear stability}\label{nonlinear}
	In this section, we consider the nonlinear problem for $\eta = \omega + \sin2\theta$, which is given by
	\begin{equation}\label{netav}
		\left\{\begin{split}
			&\partial_{t}\eta = -L\eta + \partial_{\theta}v\eta - v\partial_{\theta}\eta,\\
			&\partial_{\theta}v = H\eta,\\
			&\eta(0,\theta)=\eta_{0}(\theta),\quad v(t,0) = 0.
		\end{split}\right.
	\end{equation}
	Recall that the discussion in \cite{LL2020} shows that linearized equation \ref{netav} has the exponential decay $\|\eta\|_{\mathcal{H}_{DW}} \leq C\|\eta_0\|_{\mathcal{H}_{DW}}e^{-\frac{3}{8}t}$ for $t\geq 0$,  given initial data of the form $\eta_0 = \sum_{k\geq 1}a_{2k}e_{2k}^{(o)}.$ In this section, we establish the global well-posedness of the system \ref{netav} for initial data of the form $\sum_{k\geq 1}a_{2k}e_{2k}^{(o)}$.
	To estimate the nonlinear term, we first present the following lemma.
	\begin{lemma}\label{lefs}
		suppose that $\rho^{1/2}\partial_{\theta}f\in L^{2}$, f is odd and $f(0)=f(\pi)=0$. We have
		\begin{align}
			\|\frac{f}{\sin\theta}\|_{L^{\infty}}\lesssim\ \|\frac{\partial_{\theta}f}{\sin\theta}\|_{L^{2}}.
		\end{align}
	\end{lemma}
	\begin{proof}[Proof of Lemma \ref{lefs}] \, Since $f$ is odd, here we only need to estimate this norm in $[0,\pi]$. For $\theta\in [0,\frac{\pi}{2}],$ applying H\"{o}lder's inequality, we obtain
		\begin{align}
			\big|\frac{f}{\sin\theta}\big|& = \big|\frac{1}{\sin\theta}\int_{0}^{\theta}\partial_{\theta}f d\theta\big|\nonumber\\
			& \lesssim  \frac{1}{|\sin\theta|}\left(\int_0^{\theta}\sin^2\theta d\theta\right)^{1/2}\big\|\frac{\partial_{\theta}f}{\sin\theta}\big\|_{L^{2}}\\
			& \lesssim \big\|\frac{\partial_{\theta}f}{\sin\theta}\big\|_{L^{2}}.\nonumber
		\end{align}
		For $\theta\in (\frac{\pi}{2},\pi],$ we similarly have
		\begin{align}
			\big|\frac{f}{\sin\theta}\big|& = \big|\frac{1}{\sin\theta}\int_{\pi}^{\theta} \partial_{\theta}fd\theta\big|\nonumber\\
			&\lesssim \frac{1}{|\sin\theta|}\left(\int_{\theta}^{\pi}\sin^2\theta d\theta\right)^{1/2}\big\|\frac{\partial_{\theta}f}{\sin\theta}\big\|_{L^{2}}\\
			&\lesssim\big\|\frac{\partial_{\theta}f}{\sin\theta}\big\|_{L^{2}}.\nonumber
		\end{align}
		Since $f$ is odd, we can finish the proof of Lemma \ref{lefs}.
	\end{proof}
	
	Now, we consider the initial data of the form $\eta_0 = \sum_{k\geq 1}a_{2k}e_{2k}^{(o)}\in \mathcal{H}_{DW}$.
	\begin{proof}[Proof of Theorem \ref{wellp}]
		The discussion in \cite{LL2020} implies that
		\begin{align}\label{decl}
			\left\langle-L\eta, \eta\right\rangle_{\rho} \leq -\frac{3}{8}\left\langle\eta,\eta\right\rangle_{\rho}.
		\end{align}
		Taking the weighted $\rho$-inner product with $\eta$ on the both sides of the first equation of the system \ref{netav}, we obtain
		\begin{align}\label{dtineta}
			\frac{1}{2}\frac{d}{dt}\left\langle\eta,\eta\right\rangle_{\rho} &= \left\langle-L\eta, \eta\right\rangle_{\rho} + \left\langle \partial_{\theta}v\eta - v\partial_{\theta}\eta, \eta\right\rangle_{\rho}\\
			&\leq\ -\frac{3}{8}\left\langle\eta,\eta\right\rangle_{\rho} + \left\langle \partial_{\theta}v\eta - v\partial_{\theta}\eta, \eta\right\rangle_{\rho}.\nonumber
		\end{align}
		The second term on the right hand side of \ref{dtineta} can be written as	
		\begin{align*}
			\left\langle \partial_{\theta}v\eta - v\partial_{\theta}\eta, \eta\right\rangle_{\rho} & = \frac{1}{4\pi}\int_{\mathbb{S}^1}\frac{(\eta \partial_{\theta}^{2}v-v\partial_{\theta}^{2}\eta)\partial_{\theta}\eta}{\sin^2\theta}d\theta\\
			& =\frac{1}{4\pi}\int_{\mathbb{S}^1}\frac{\eta \partial_{\theta}^{2}v \partial_{\theta}\eta}{\sin^2\theta}d\theta - \frac{1}{4\pi}\int_{\mathbb{S}^1}\frac{v \partial_{\theta}^{2}\eta \partial_{\theta}\eta}{\sin^2\theta}d\theta\\
			& \equiv I + II.		
		\end{align*}
		Direct estimates give
		\begin{align}\label{eI}
			| I| & \lesssim\ \|\frac{\eta}{\sin\theta}\|_{L^{\infty}}\|\rho^{1/2}\partial_{\theta}\eta\|_{L^{2}}\| \partial_{\theta}^{2}v\|_{L^2}\\
			&\lesssim\ \|\rho^{1/2}\partial_{\theta}\eta\|_{L^{2}}^{3}.\nonumber
		\end{align}
		where we have used Lemma \ref{lefs} and the fact
		\begin{align*}
			\|\partial_{\theta}^{2}v\|_{L^{2}}\lesssim \|\partial_{\theta}\eta\|_{L^{2}}\leq\|\rho^{1/2}\partial_{\theta}\eta\|_{L^{2}}.
		\end{align*}
		Rewrite the term $II$ as
		\begin{align}\label{II}
			II &= -\frac{1}{8\pi}\int_{\mathbb{T}}\frac{v \partial_{\theta}((\partial_{\theta}\eta)^{2})}{\sin^{2}\theta}d\theta\nonumber\\
			&= \frac{1}{8\pi}\int_{\mathbb{T}}\frac{\partial_{\theta}v (\partial_{\theta}\eta)^{2}}{\sin^{2}\theta}d\theta + \frac{1}{8\pi}\int_{\mathbb{T}}v (\partial_{\theta}\eta)^{2} \partial_{\theta}\left(\frac{1}{\sin^{2}\theta}\right)d\theta\\
			& =: II_{1} + II_{2}.\nonumber
		\end{align}
		Direct estimates show
		\begin{align}\label{eII1}
			| II_{1}| &\lesssim\ \| \partial_{\theta}v\|_{L^{\infty}}\|\rho^{1/2}\partial_{\theta}\eta\|_{L^{2}}^{2}\nonumber\\
			&\lesssim\ \|\rho^{1/2}\partial_{\theta}\eta\|_{L^{2}}^{2}\|\partial_{\theta}\eta\|_{L^2}\\
			&\lesssim\ \|\rho^{1/2}\partial_{\theta}\eta\|_{L^{2}}^{3}.\nonumber
		\end{align}
		
		For $\|\frac{v}{\sin\theta}\|_{L^{\infty}}$, since $v$ is odd and periodic, we have $v(\pi)=v(0)=0$ and only need to estimate this norm in $[0,\pi]$. Since $\sin\theta\geq\frac{2}{\pi}\min\{\theta,\pi-\theta\}$ in $[0,\pi]$, we have $\|\frac{v}{\sin\theta}\|_{L^{\infty}}\lesssim\|\partial_{\theta}v\|_{L^{\infty}}$ by Lagrange's mean value theorem. Then
		\begin{align}\label{eII2}
			| II_{2}| &\ = |\frac{1}{4\pi}\int_{\mathbb{T}}v (\partial_{\theta}\eta)^{2}\ \frac{\cos\theta}{\sin^{3}\theta}d\theta|\nonumber\\
			&\lesssim\ \|\frac{v}{\sin\theta}\|_{L^{\infty}}\|\rho^{1/2}\partial_{\theta}\eta\|_{L^{2}}^{2}\nonumber\\
			&\lesssim\|\partial_{\theta}v\|_{L^{\infty}}\|\rho^{1/2}\partial_{\theta}\eta\|_{L^{2}}^{2}\\
			&\lesssim\|\partial_{\theta}\eta\|_{L^{2}}\|\rho^{1/2}\partial_{\theta}\eta\|_{L^{2}}^{2}\nonumber\\
			&\lesssim\ \|\rho^{1/2}\partial_{\theta}\eta\|_{L^{2}}^{3}.\nonumber
		\end{align}
		Substituting \ref{eI}-\ref{eII2} into \ref{dtineta}, we deduce
		\begin{align}
			\frac{d}{dt}\left\langle\eta,\eta\right\rangle_{\rho}\ \leq\ -\frac{3}{4}\left\langle\eta,\eta\right\rangle_{\rho}^{2}\ + C\left\langle\eta,\eta\right\rangle_{\rho}^{3},
		\end{align}
		where C is a positive constant. Consequently, there exists an absolute constant $\varepsilon >0$ such that if $\left\langle\eta_{0},\eta_{0}\right\rangle_{\rho}\leq\varepsilon$, using a bootstrap argument, we obtain
		\begin{align*}
			\left\langle\eta,\eta\right\rangle_{\rho}\lesssim\varepsilon,
		\end{align*}
		for all $t > 0$. This completes the proof of Theorem \ref{wellp}.
	\end{proof}
	
	\section{appendix}\label{num}
	{\bf A.~~}For the parameters $\lambda_{k}^1, \lambda_{k}^2, a_{k}$ ($k\geq 1$) in Lemma \ref{lepsk}, we have the following estimates.
	\begin{lemma}\label{lemmabound}
		There exist two absolute constants $\lambda_{2}>\lambda_{1}>0$ satisfying
		\begin{align}
			\frac{1}{50}< \lambda_{1}< \lambda_{k}^1, \lambda_{k}^2, a_{k} < \lambda_{2}<\frac{3}{5}
		\end{align}
		for all $k\geq 1$.
	\end{lemma}
	\begin{proof}[Proof of Lemma \ref{lemmabound}]
		The exact expression for $-d_{k+2} + d_{k}$ is given by
		\begin{align*}
			-d_{k+2} + d_{k} & = -\frac{k^{2}(k+4)}{4(k+2)^{2}} + \frac{(k-2)^{2}(k+2)}{4k^{2}}\\
			&= -\frac{k^{4}+4k^{3}+8k^{2}-8k-16}{2(k+2)^{2}k^{2}}\\
			&= -\frac{1}{2} - \frac{2k^{2}-4k-8}{(k+2)^{2}k^{2}}.
		\end{align*}
		Let
		\begin{align}\label{1225}
			f(x) =-\frac{1}{2}-2\frac{x^2 -2x-4}{(x+2)^2 x^2}, \quad x\geq 1.
		\end{align}
		Then we have
		\begin{align*}
			f'(x) = \frac{4x(x+2)(x^3-3x^2-10x-8)}{(x+2)^4 x^2}.
		\end{align*}
		It is direct to obtain  that there exists a unique real number $x_0\in (5,6)$ such that $f'(x_0)=0$. Moreover, we have that $f'(x)\leq 0$ if $x\in [1,x_0],$ and $f'(x)\geq 0$ if $x\in [x_0,\infty)$. Direct computations shows that
		$f(1) = -d_3 + d_1 =\frac{11}{18} > 0, f(2) = -d_4 + d_2 =-\frac{3}{8} < 0$, and
		\begin{align*}
			f(5) = -\frac{1269}{2450} < f(6) = -\frac{149}{288} < 0.
		\end{align*}
		It concludes that
		\begin{align*}
			f(5) \leq f(k) \leq -\frac{1}{2}
		\end{align*}
		for $k\geq 4$, where the second inequality is from \eqref{1225}. It follows that
		\begin{align}
			\frac{3}{8} \leq |-d_{k+2} + d_{k}| \leq \frac{11}{18}
		\end{align}
		for $k\geq 1$. Furthermore, using the expression
		\begin{align*}
			a_{k} = \left(-d_{k+2} + d_{k}\right)^{2} = \left(-\frac{1}{2} - \frac{2k^{2}-4k-8}{(k+2)^{2}k^{2}}\right)^2,
		\end{align*}
		we have
		\begin{align}\label{12252}
			a_{k} \geq \frac{1}{4}
		\end{align}
		for $k\ge 4$, and
		\begin{align}\label{12253}
			(\frac{3}{8})^2 \leq a_k \leq (\frac{11}{18})^2
		\end{align}
		for $k\ge 1$.
		The explicit expression for $\varepsilon_{k}$ is given by
		\begin{align*}
			\varepsilon_{k} &= -2d_{k+2}^{2}+d_{k}d_{k+2}+ d_{k+2}d_{k+4}\\
			& = \frac{-2k^{4}(k+4)^{2}}{4^{2}(k+2)^{4}}+\frac{(k-2)^{2}(k+4)}{4^{2}(k+2)} + \frac{k^{2}(k+6)}{4^{2}(k+4)}\\
			& = \frac{-2k^{3}+32k+32}{(k+2)^{4}(k+4)}.
		\end{align*}
		Direct computations reveal that
		\begin{align}
			|\varepsilon_{k}| \leq \varepsilon_{1} = \frac{62}{405}
		\end{align}
		for all $k\geq 1$. The analysis of $\varepsilon_{k}$ is similar to that for $-d_{k+2} + d_k$, so we omit the details here.
		
		Moreover, we have that $-d_{k+2} + d_{k}\rightarrow -\frac{1}{2}, a_{k}\rightarrow \frac{1}{4}$ and $\varepsilon_{k}\rightarrow 0$ as $k\rightarrow \infty$.
		
		Now we estimate the eigenvalues $\lambda_{k}^1, \lambda_{k}^2$ of the matrix $A_k$, which  are given by
		\begin{align*}
			\lambda_{k}^{1}=\frac{a_{k}+a_{k+2}-\sqrt{(a_{k}-a_{k+2})^{2}+4\varepsilon_{k}^{2}}}{2},
		\end{align*}
		\begin{align*}
			\lambda_{k}^{2}=\frac{a_{k}+a_{k+2}+\sqrt{(a_{k}-a_{k+2})^{2}+4\varepsilon_{k}^{2}}}{2}
		\end{align*}
		for $k\geq 1$.
		From \eqref{12252}, it yields
		\begin{align*}
			a_{k}a_{k+2} \geq \min\{a_{1}a_{3}, a_{2}a_{4}, a_{3}a_{5}, \frac{1}{16}\} = a_{2}a_{4} = (-\frac{3}{8})^2(-\frac{37}{72})^2 > \frac{1}{100},
		\end{align*}
		which implies
		\begin{align}\label{12255}
			\lambda_{k}^{1} &= \frac{2(a_k a_{k+2} - \varepsilon_{k}^2)}{a_k + a_{k+2} + \sqrt{(a_{k}-a_{k+2})^{2}+4\varepsilon_{k}^{2}}}\nonumber\\
			&\geq \frac{2(a_{2}a_{4} - \varepsilon_{1}^2)}{2a_1 + \sqrt{(a_{1}-a_{2})^{2}+4\varepsilon_{1}^{2}}} = \frac{2[(-\frac{3}{8})^2(-\frac{37}{72})^2 - (\frac{62}{405})^2]}{2(\frac{11}{18})^2 +  \sqrt{[(\frac{11}{18})^2- (-\frac{3}{8})^2]^2 + 4(\frac{62}{405})^2}} > \frac{1}{50}
		\end{align}
		for $k\geq 1$. Similarly, it deduces that
		\begin{align}\label{12256}
			\lambda_{k}^{1} < \lambda_{k}^2 & = \frac{a_{k}+a_{k+2}+\sqrt{(a_{k}-a_{k+2})^{2}+4\varepsilon_{k}^{2}}}{2}\nonumber\\
			& \leq \frac{2 a_{1} + \sqrt{(a_{1}-a_{2})^{2}+4\varepsilon_{1}^{2}}}{2} = \frac{2(\frac{11}{18})^2 +  \sqrt{[(\frac{11}{18})^2- (-\frac{3}{8})^2]^2 + 4(\frac{62}{405})^2}}{2}< \frac{3}{5}
		\end{align}
		for $k\geq 1$.
		
		In view of  \eqref{12253}, \eqref{12255} and \eqref{12256}, there exist two absolute constants $\lambda_{2} > \lambda_{1}> 0$ such that
		\begin{align*}
			\frac{1}{50} < \lambda_{1} < \lambda_{k}^{1}, \lambda_{k}^{2}, a_{k} < \lambda_{2}<\frac{3}{5}
		\end{align*}
		for all $k\geq 1$.
		The proof of the lemma is finished.
	\end{proof}

	{\bf B.~~} In the end of the paper, we present the numerical illustrations of $d_{k} - d_{k+2}, a_{k},\varepsilon_{k}$ and $\lambda_{k}^1,\lambda_{k}^2$. Our rigorous result in Lemma \ref{lemmabound} is consistent with the numerical result as follows.	
	\begin{figure}[H]\centering
		\subfigure[Numerical illustration of $d_{k}-d_{k+2}$.]{
			\centering
			\includegraphics[width=0.42\linewidth]{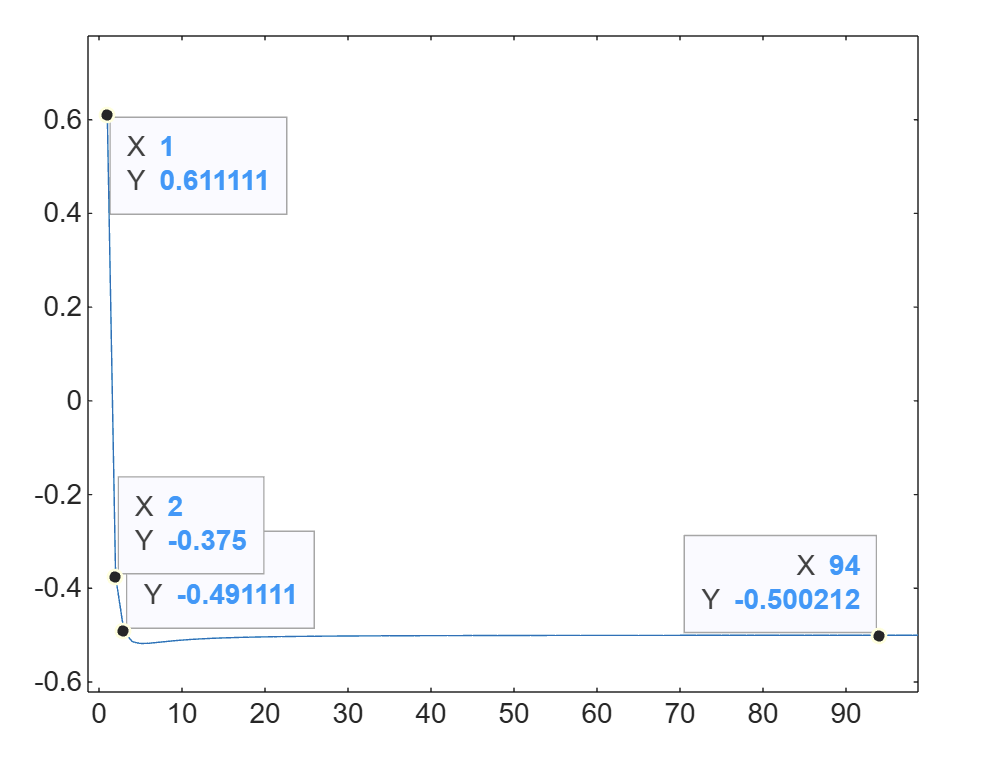}
		}
		\subfigure[Numerical illustration of $a_{k}$.]{\centering
			\includegraphics[width=0.44\linewidth]{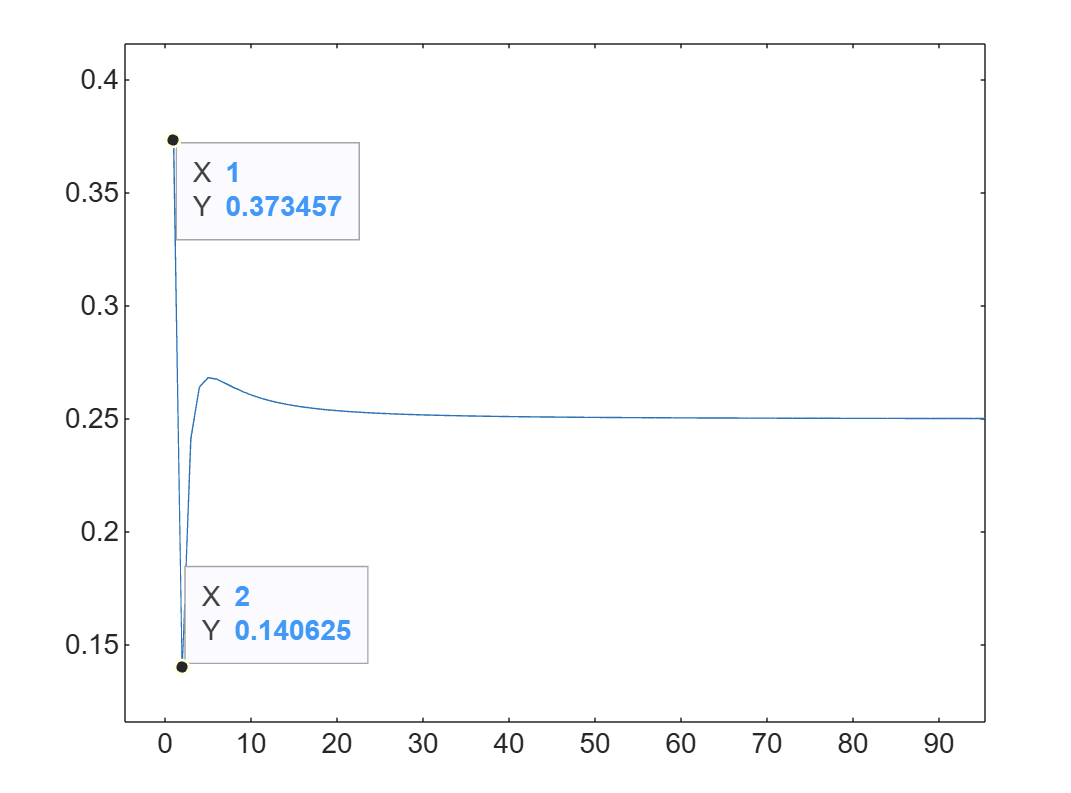}}
		\subfigure[Numerical illustration of $\varepsilon_{k}$.]{\centering
			\includegraphics[width=0.42\linewidth]{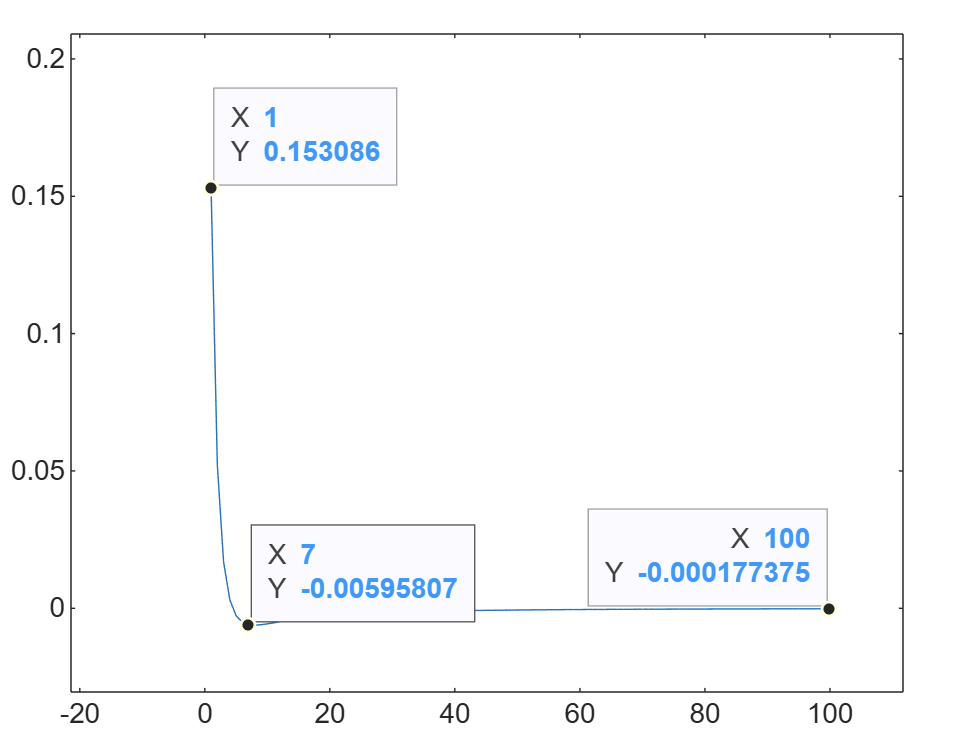}}
	\end{figure}

	\begin{figure}[H]\centering
		\subfigure[Numerical illustration of $\lambda_{k}^{1}$.]{\centering
			\includegraphics[width=0.44\linewidth]{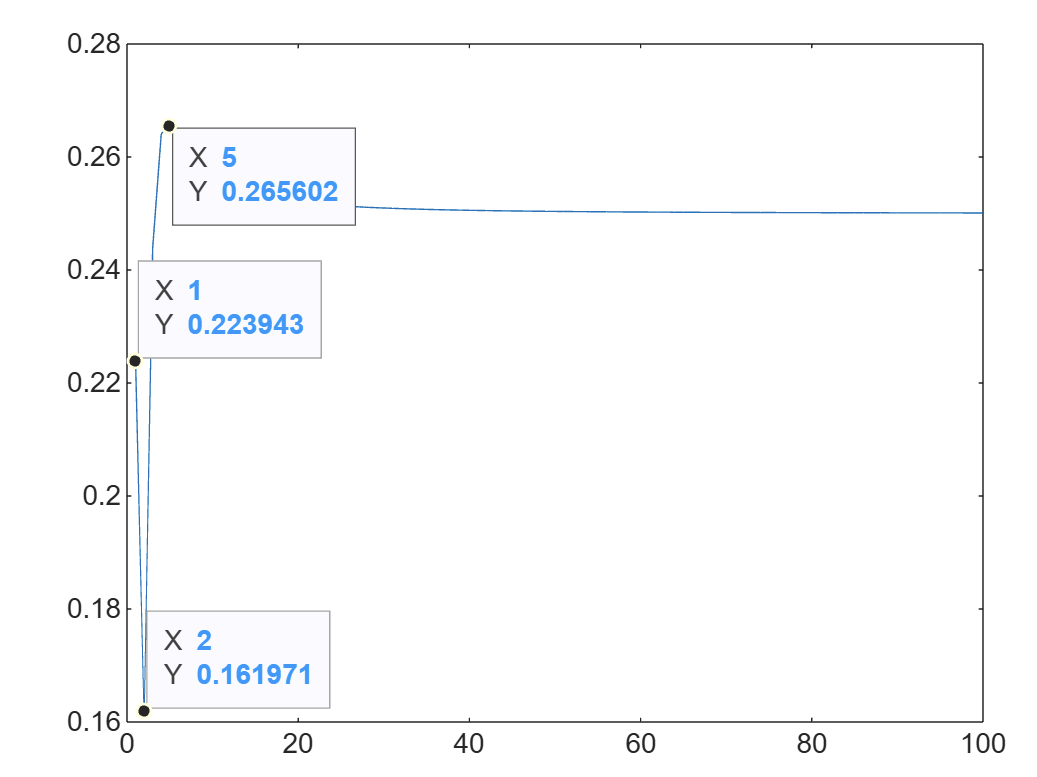}
		}
		\subfigure[Numerical illustration of $\lambda_{k}^{2}$.]{\centering
			\includegraphics[width=0.47\linewidth]{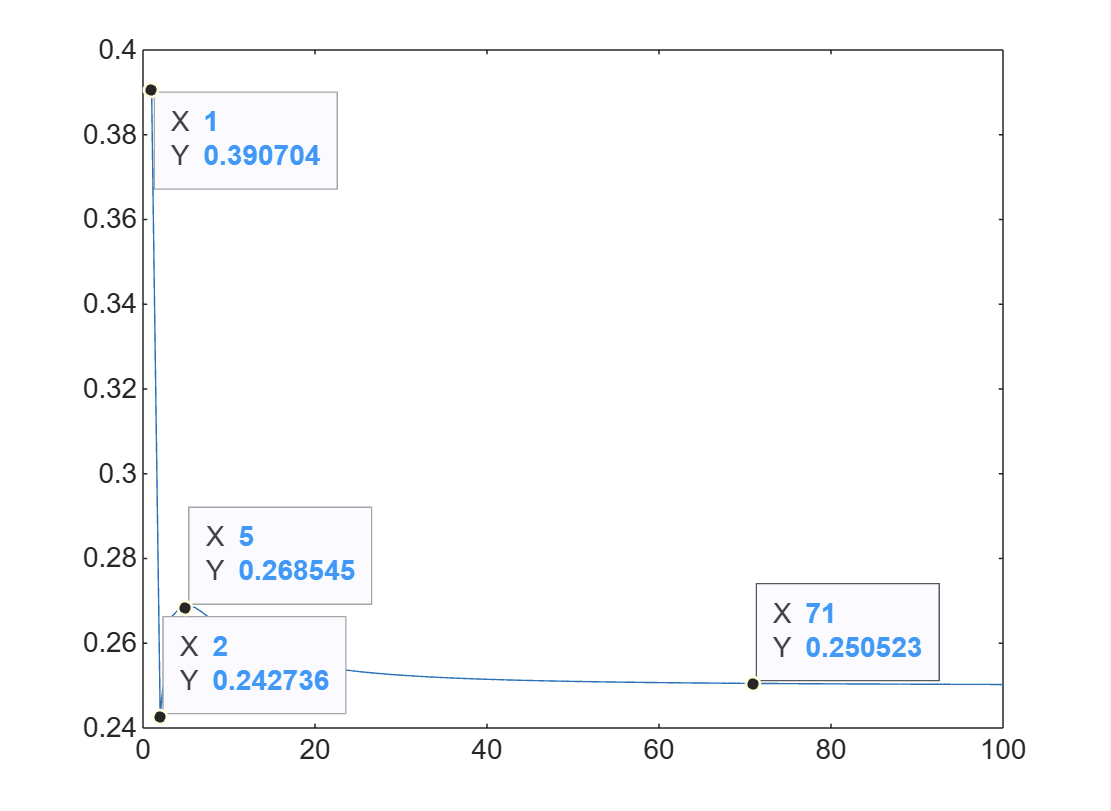}
		}
	\end{figure}
	Here
	\begin{align*}
		\lambda_{k}^{1}=&-\frac{\sqrt{{\left({\left(\frac{k^2\,\left(k+4\right)}{4\,{\left(k+2\right)}^2}-\frac{{\left(k+2\right)}^2\,\left(k+6\right)}{4\,{\left(k+4\right)}^2}\right)}^2-{\left(\frac{{\left(k-2\right)}^2\,\left(k+2\right)}{4\,k^2}-\frac{k^2\,\left(k+4\right)}{4\,{\left(k+2\right)}^2}\right)}^2\right)}^2+4\,{\left(\frac{k^2\,\left(k+6\right)}{16\,\left(k+4\right)}+\frac{{\left(k-2\right)}^2\,\left(k+4\right)}{16\,\left(k+2\right)}-\frac{k^4\,{\left(k+4\right)}^2}{8\,{\left(k+2\right)}^4}\right)}^2}}{4}\\
		&+\frac{{\left(\frac{k^2\,\left(k+4\right)}{4\,{\left(k+2\right)}^2}-\frac{{\left(k+2\right)}^2\,\left(k+6\right)}{4\,{\left(k+4\right)}^2}\right)}^2}{2}+\frac{{\left(\frac{{\left(k-2\right)}^2\,\left(k+2\right)}{4\,k^2}-\frac{k^2\,\left(k+4\right)}{4\,{\left(k+2\right)}^2}\right)}^2}{2},
	\end{align*}
	\begin{align*}
		\lambda_{k}^{2}=&\frac{\sqrt{{\left({\left(\frac{k^2\,\left(k+4\right)}{4\,{\left(k+2\right)}^2}-\frac{{\left(k+2\right)}^2\,\left(k+6\right)}{4\,{\left(k+4\right)}^2}\right)}^2-{\left(\frac{{\left(k-2\right)}^2\,\left(k+2\right)}{4\,k^2}-\frac{k^2\,\left(k+4\right)}{4\,{\left(k+2\right)}^2}\right)}^2\right)}^2+4\,{\left(\frac{k^2\,\left(k+6\right)}{16\,\left(k+4\right)}+\frac{{\left(k-2\right)}^2\,\left(k+4\right)}{16\,\left(k+2\right)}-\frac{k^4\,{\left(k+4\right)}^2}{8\,{\left(k+2\right)}^4}\right)}^2}}{4}\\
		&+\frac{{\left(\frac{k^2\,\left(k+4\right)}{4\,{\left(k+2\right)}^2}-\frac{{\left(k+2\right)}^2\,\left(k+6\right)}{4\,{\left(k+4\right)}^2}\right)}^2}{2}+\frac{{\left(\frac{{\left(k-2\right)}^2\,\left(k+2\right)}{4\,k^2}-\frac{k^2\,\left(k+4\right)}{4\,{\left(k+2\right)}^2}\right)}^2}{2}.
	\end{align*}

\end{document}